\documentclass[12pt,reqno]{amsart}
\usepackage{amsthm,amsfonts,amssymb,euscript}

\def\rev#1{\frac{1}{#1}}

\def\ddb#1{\sqrt{-1}\partial\bar{\partial}#1}

\newcommand{\R}{{\mathbb R}}
\newcommand{\C}{{\mathbb C}^{n}}

\newcommand{\B}{{\mathcal{B}}}

\newcommand{\ra}{{\rightarrow}}

\newcommand{\dr}{\omega}
\newcommand{\ka}{K\"{a}hler}

\theoremstyle{plain}
  \newtheorem{theorem}{Theorem}[section]
  \newtheorem{proposition}[theorem]{Proposition}
  \newtheorem{lemma}[theorem]{Lemma}

  \newtheorem{remark}[theorem]{Remark}

\theoremstyle{definition}
  \newtheorem{definition}[theorem]{Definition}

\numberwithin{equation}{section}
\begin{document}

\title[Conical K\"{a}hler-Einstein metrics]{Conic K\"{a}hler-Einstein metrics along simple normal crossing divisors on Fano manifolds}
\author{Aijin Lin}
\author{Liangming Shen}
\address{College of Science, National University of Defense Technology, Changsha, Hunan, P.R.China, 410073}
\email{aijinlin@pku.edu.cn}
\address{Department of Mathematics, The University of British Columbia, Vancouver, B.C., Canada V6T 1Z2}
\email{lmshen@math.ubc.ca}

\begin{abstract}

We prove that on one \ka-Einstein Fano manifold without holomorphic vector fields, there exists a unique conical \ka-Einstein metric along a simple normal crossing divisor with admissible prescribed cone angles. We also establish a curvature estimate for conic metrics along a simple normal crossing divisor which generalizes Li-Rubinstein's curvature estimate for one divisor case.

\end{abstract}

\maketitle


\section{Introduction}

Conic \ka\ metrics are very useful in the study of \ka\ geometry. Recently there are a lot of works on this topic \cite{Ber} \cite{Br} \cite{CGP} \cite{Do} \cite{GP} \cite{Je} \cite{JMR} \cite{LS} \cite{Ru} \cite{Sh1} \cite{SW} \cite{Ti14} \cite{TZ} \cite{Yao}. It played a crucial role in the solution to Yau-Tian-Donaldson Conjecture \cite{Ti12} \cite{CDS1}
\cite{CDS2} \cite{CDS3} on Fano manifolds. In this solution one important part is the analysis of the continuity path of the cone angle, which has been studied in \cite{Do}.
The idea is to investigate the behavior of the solutions to conical \ka-Einstein metrics as the angle tends to 1. Then the stability condition guarantees that the solution could be extended, which generates a smooth \ka-Einstein metric. On the other hand, we can consider the existence of conical \ka-Einstein metrics assuming the existence of smooth \ka-Einstein metrics, which could be studied along the continuity path of decreasing cone angles. Related works in this direction could be found in \cite{LS} \cite{SW} \cite{TZ} in case of one smooth divisor. In \cite{SW}, Song-Wang also considered the case of simple normal crossing divisors on toric Fano manifolds. In this paper, we consider the general situation of simple normal crossing divisors and our main result is as below, which could be thought as the generalization of Theorem 1.1 of \cite{LS}:
\begin{theorem}\label{mainthm}
Given a Fano manifold $(M,\dr_{0})$ without holomorphic vector fields where $[\dr_{0}]=c_{1}(M),$ which admits a \ka-Einstein metric $\dr_{KE}$ satisfying $Ric(\dr_{KE})=\dr_{KE}.$ For one simple normal crossing divisor
$D=\sum\limits_{r=1}^{m}D_{r}$ where each $D_{r}$ is semi-ample, and a sequence of positive rational numbers $c_{1},\cdots,c_{m}$ which satisfy that
\begin{equation}\label{eq:snc divisor}
\sum_{r=1}^{m}c_{r}[D_{r}]=c_{1}(M).
\end{equation}
If for all $r=1,\cdots,m$ it holds that $c_{r}>1$  or
$$\lambda_{r}:=\inf\{\lambda>0|\lambda K_{M}^{-1}-[D_{r}]>0\}\geq\frac{n}{n+1}$$ when $c_{r}=1,$
then for $\mu,\beta_{1},\cdots,\beta_{m}\in (0,1)$ which satisfy
\begin{equation}\label{eq:angle}
1-\beta_{r}=(1-\mu)c_{r},
\end{equation}
there exists a unique $C^{2,\alpha,\B}$ conical \ka-Einstein metric in $c_{1}(M)$ with cone angle $2\pi\beta_{r}$ along each irreducible divisor
$D_{r}$ for $r=1,\cdots,m$ for some $\alpha\in (0,1)$ depending on $\B:=(\beta_{1},\cdots,\beta_{m}).$
\end{theorem}

We recall that a \ka\ $(1,1)$-current $\dr$ is  a
conic metric along the simple normal divisor $D=\sum\limits_{r=1}^{m}D_{r}$ with cone angle $\beta_{r}$ along each
irreducible divisor $D_{r}$ if it is smooth outside the divisor $D$ and equivalent to the
standard conic metric
\begin{equation}\label{eq:def-conic}
\dr_{cone}=\sqrt{-1}\left(\sum_{r=1}^{m}\frac{dz_{r}\wedge
d\bar{z}_{r}}{|z_{r}|^{2(1-\beta_{r})}}+\sum_{r=m+1}^{n}dz_{r}\wedge d\bar{z}_{r}\right)
\end{equation}
The study of such metrics is related to klt singularities in algebraic geometry. However in this general setting even the linear theory developed in \cite{Do} \cite{JMR} is unclear to hold so the continuity method and a prior $C^{0}$-estimate in \cite{JMR} for smooth divisor case cannot be adapted directly. Alternately, we will make use the approximation method developed in \cite{Ti12} \cite{Sh1} \cite{TZ}, to construct smooth solutions to the complex Monge-Ampere equations which approximate the conical Monge-Ampere equation associated to the conical \ka-Einstein metric along a simple normal crossing divisor. To obtain the solutions, we need to derive the properness of perturbed energy functionals, e.g.,
Ding energy or Mabuchi energy from the existence of smooth \ka-Einstein metric. We also need to establish a uniform $C^{0}$-estimate for the approximating solutions and then deduce a unique weak solution to conical Monge-Ampere equation by Berndtsson's uniqueness theorem \cite{Bern}.

However, the approximation method can only guarantee the existence of weak solution to conical Monge-Ampere equation. To establish the regularity of the solution we first need to establish the Laplacian estimate of the solution. In case of one irreducible divisor \cite{JMR} this estimate comes from
an application of Chern-Lu's Inequality \cite{Lu} which requires a bisectional upper bound estimate of the background conic metric by Li-Rubinstein (see the appendix of \cite{JMR} or C. Li's thesis \cite{Li}). In simple normal crossing case the curvature estimate is much more complicated. In \cite{CGP} \cite{GP} Campana, Guenancia and Paun construct an approximating sequence of background conic metric. As their metrics do not have a uniformly curvature bound from any side they need more complicated calculations to derive the Laplacian estimate. Alternately in \cite{DS} Datar-Song gave a simple Laplacian estimate depending on Li-Rubinstein's curvature estimate, which we will use in this paper. Actually we note that if the linear theory for multiple divisors case is established we could still derive a Laplacian estimate by generalizing Li-Rubinstein's curvature estimate:

\begin{theorem}\label{thm-upper}
 For a \ka\ manifold $M$ with a simple normal crossing divisor consisting of $m$ irreducible divisors $D_{r},r=1,\cdots,m$ on $M,$ suppose we have
a conic metric $$\dr=\dr_{0}+\sum\limits_{r=1}^{m}\ddb||S_{r}||_{r}^{2\beta_{r}}$$ with cone angle $2\pi\beta_{r}$ along each irreducible divisor $D_{r}.$ Then we have two cases:\\
(1)for either all cone angles $\beta_{r}\leq\frac{1}{2}$ or $D$ is composed by irreducible divisors free of triple singularities,  the bisectional curvature of $\dr$ is uniformly bounded from above on $M\setminus D.$ \\
(2)for $D$ containing higher multiple singularities, there exists a smooth potential function $\varphi_{0}$ such that $\dr+\ddb\varphi_{0}$ is a conic \ka\ metric which is equivalent to $\dr$ and its bisectional curvature is uniformly bounded from above on $M\setminus D.$
\end{theorem}

\begin{remark}
By Y. Rubinstein and the author's computation, until now we can only obtain the upper bisectional curvature bound for the case that either all cone angles $\beta_{i}\leq\frac{1}{2}$ or $D$ is composed by irreducible divisors free of triple singularity, which is just the first part of \ref{thm-upper}. The problem for higher multiple singularities is that near such singularities, the diagonal terms of the background metric $\dr_{0}$ cannot supply enough control over its non-diagonal terms. However, if we add a smooth function $\varphi_{0}$ to enlarge the diagonal terms then such control can be achieved. In this sense the two conclusions could be combined as we can set $\varphi_{0}=0$ in the first case. Although this setting changes the original background metric, as the modified metric is still equivalent to the original one, the Laplacian estimate can still be established by Chern-Lu's Inequality under the settings of \cite{JMR}.
\end{remark}

Finally, the $C^{2,\alpha,\B}$-estimate could be derived by Tian's beautiful estimate in \cite{Ti14}. The simple normal crossing case could
also be derived by some generalizations for which we will give main steps and the readers can see \cite{Sh3} which established a parabolic estimate in more details.\\

\noindent{\bf Acknowledgment.} First of all the authors want to thank their advisor Professor Gang Tian for suggesting this problem and a lot of guidance and encouragement. They want to thank Professor Xiaohua Zhu for a lot of discussions on conic metric problems. Moreover the second author wants to thank Professor Zhiqin Lu, Rafe Mazzeo, Yanir Rubinstein and Xiaowei Wang for very help discussions on this work. The first author was partially supported by the NSFC Grant 11401578 and part of this work is based upon work supported by the National Science Foundation under Grant No. DMS-1440140 while the second author was in residence at the Mathematical Sciences Research Institute in Berkeley, California, during the Spring 2016 semester. Finally the authors want to thank the referee for his useful suggestions.

\section{Basic setting, energy functionals and log $\alpha$-invariants}

As \cite{Ti12}, first we need to derive a complex Monge-Ampere equation from the equation \eqref{eq:snc divisor} and
conical \ka-Einstein condition. We know that a simple normal crossing divisor $D=\sum\limits_{r=1}^{m}D_{r}$ is consisting of $m$ irreducible divisors
which can be written as $D_{r}=\{z_{r}=0\}$ locally. Denote $S_{r}$ as the defining holomorphic section of $D_{r}$ and $||\cdot||_{r}$ as the Hermitian
metric defined on the holomorphic line bundle $[D_{r}].$ By \eqref{eq:snc divisor} \eqref{eq:angle} we have that
\begin{equation}\label{coho-divisor}
\sum_{r=1}^{m}(1-\beta_{r})R(||\cdot||_{r})=(1-\mu)(\dr_{0}+\ddb h),
\end{equation}
where $R(||\cdot||_{r})$ is the curvature of $||\cdot||_{r}$ and $h$ is a smooth function.
Suppose $\dr_{c}=\dr_{0}+\ddb\varphi$ is the required conical \ka-Einstein metric, then by equations above it satisfies
\begin{equation}\label{cke}
Ric(\dr_{c})=\mu\dr_{c}+\sum_{r=1}^{m}2\pi(1-\beta_{r})[D_{r}]
\end{equation}
If the background metric $\dr_{0}$ satisfies
$$Ric(\dr_{0})=\dr_{0}+\ddb h_{0},$$
where $h$ is a smooth function satisfies $\int_{M}(e^{h}-1)\dr_{0}^{n}=0,$
then from above equations we could derive a conical Monge-Ampere equation for $\varphi:$
\begin{equation}\label{eq:cke-cma}
(\dr_{0}+\ddb\varphi)^{n}=e^{h_{0}-(1-\mu)h-\mu\varphi-\sum_{r=1}^{m}(1-\beta_{r})\log||S_{r}||^{2}+a_{\mu}}\dr_{0}^{n},
\end{equation}
where $a_{\mu}$ is a constant satisfying
$$\int_{M}(e^{h_{0}-(1-\mu)h-\sum_{r=1}^{m}(1-\beta_{r})\log||S_{r}||^{2}+a_{\mu}}-1)\dr_{0}^{n}=0.$$
Note that in \cite{JMR} \cite{LS} they applied continuity method to solve such conical Monge-Ampere equation. Actually their works highly rely on the linear theory with respect to the standard conic metric along one smooth divisor. For simple normal crossing divisors such linear theory is not known yet. Alternately, we try to establish the existence of the solution by approximating method in \cite{Ti12} \cite{Sh1} \cite{TZ}. Basically, we can perturb the equation \eqref{eq:cke-cma} to the following smooth equation:
\begin{equation}\label{eq:app-cke-cma}
 (\dr_{0}+\ddb\varphi_{\epsilon})^{n}=e^{h_{0}-(1-\mu)h-\mu\varphi_{\epsilon}-\sum_{r=1}^{m}(1-\beta_{r})
 \log(||S_{r}||^{2}+\epsilon)+a_{\mu,\epsilon}}\dr_{0}^{n},
\end{equation}
where $a_{\mu,\epsilon}$ is a constant satisfying
$$\int_{M}(e^{h_{0}-(1-\mu)h-\sum_{r=1}^{m}(1-\beta_{r})\log(||S_{r}||^{2}+\epsilon)+a_{\mu,\epsilon}}-1)\dr_{0}^{n}=0.$$
Our strategy is to solve this approximating equation and try to obtain the compactness of the approximating solutions. To achieve this goal, as \cite{Ti97} \cite{Ti12} we will introduce corresponding energy functionals,
analyze the properness, and finally establish the uniform $C^{0}$-estimate.

Recall in \cite{Di} \cite{DT} \cite{Ti97}, we have the following energy functionals for $\varphi\in PSH(M,\dr_{0})$:
\begin{definition}
 $$ (1) J_{\dr_{0}}(\varphi)=\rev V\sum_{i=0}^{n-1}\frac{i+1}{n+1}\int_{M}\sqrt{-1}\partial\varphi\wedge\bar{\partial}
\varphi\wedge\dr_{0}^{i}\wedge\dr_{\varphi}^{n-i-1},$$\\
 $$ (2) I_{\dr_{0}}(\varphi)=\rev V\int_{M}\varphi(\dr_{0}^{n}-\dr_{\varphi}^{n}),$$
where $V=\int_{M}\dr_{0}^{n}$, $\dr_{\varphi}=\dr_{0}+\ddb\varphi.$
\end{definition}
There are some nice properties of those functionals, see \cite{Ti00} for more details. Then we can introduce
the twisted Ding functional and the twisted Mabuchi functional (see \cite{JMR} \cite{LS} \cite{Sh1}), which are the Lagrangians of the conical
Monge-Ampere equation \eqref{eq:cke-cma}. For simplicity here we set
$$H_{0,\mu}=h_{0}-(1-\mu)h-\sum_{r=1}^{m}(1-\beta_{r})\log||S_{r}||_{r}^{2}+a_{\mu},$$ and we can choose a family $\varphi_{t}$ connected 0 and $\varphi\in PSH(M,\dr_{0}):$
\begin{definition}\label{def:ding-ma}
 (1) We define twisted Ding functional as
\begin{equation}\label{eq:ding}
 F_{\dr_{0},\mu}(\varphi)=J_{\dr_{0}}(\varphi)-\rev V\int_{M}\varphi\dr_{0}^{n}-\frac{1}{\mu}\log\left(\rev V\int_{M}
e^{H_{0,\mu}-\mu\varphi}\dr_{0}^{n}\right),
\end{equation}
 (2) we define twisted Mabuchi functional as
\begin{align*}
 \nu_{\dr_{0},\mu}(\varphi)=&-\frac{n}{V}\int_{0}^{1}\int_{M}\dot{\varphi_{t}}(Ric(\dr_{\varphi})-\mu\dr_{\varphi}-\sum_{r=1}^{m}
2\pi(1-\beta_{r})[D_{r}])\wedge\dr_{\varphi}^{n-1}dt\\=&\rev V\int_{M}\log\frac{\dr_{\varphi}^{n}}{\dr_{0}^{n}}\dr_{\varphi}^{n}
+\rev V\int_{M}H_{0,\mu}(\dr_{0}^{n}-\dr_{\varphi}^{n})-\mu(I_{\dr_{0}}(\varphi)-J_{\dr_{0}}(\varphi))\\=&
\rev V\int_{M}\log\frac{\dr_{\varphi}^{n}}{\dr_{0}^{n}}\dr_{\varphi}^{n}+\rev V\int_{M}H_{0,\mu}(\dr_{0}^{n}-\dr_{\varphi}^{n})+
\mu(F_{\dr_{0}}^{0}(\varphi)+\rev V\int_{M}\varphi\dr_{\varphi}^{n}),
\end{align*}
where $$F_{\dr_{0}}^{0}(\varphi)=J_{\dr_{0}}(\varphi)-\rev V\int_{M}\varphi\dr_{0}^{n}.$$
\end{definition}
In \cite{Ti97}, Tian pointed out that the $C^{0}$-estimate for the \ka-Einstein problem could be established if the Ding functional or the Mabuchi functional is proper.
Similarly, to establish the uniform $C^{0}$-estimates for the approximating equation \eqref{eq:app-cke-cma}, we need to check the properness of the corresponding
Ding functionals or Mabuchi functionals, which could be derived from the properness of corresponding twisted functionals. First, we recall the definition of the properness:
\begin{definition}
 Suppose the twisted Ding functional $F_{\dr_{0},\mu}(\varphi)$(twisted Mabuchi functional $\nu_{\dr_{0},\mu}(\varphi)$) is bounded
from below, i.e. $F_{\dr_{0},\mu}(\varphi)\geq -c_{\dr_{0}}$ ($\nu_{\dr_{0},\mu}(\varphi)\geq -c_{\dr_{0}}$), we say it is proper on
$PSH(M,\dr_{0})$, if there exists an increasing function $f: [-c_{\dr_{0}},\infty)\ra\R, $ and $\lim_{t\ra\infty}f(t)=\infty$, such that for any
$\varphi\in PSH(M,\dr_{0}),$ we have $$F_{\dr_{0},\mu}(\varphi)\geq f(J_{\dr_{0}}(\varphi))\quad(\nu_{\dr,\mu}(\varphi)\geq f(J_{\dr_{0}}(\varphi))).$$
\end{definition}
There are a lot of properties of these functionals, for more details, see \cite{Ti00} \cite{LS}. The main result of this section is the properness of the twisted Ding
functional assuming the conditions in Theorem \ref{mainthm}:
\begin{theorem}\label{thm-cma-proper}
Assume the conditions of Theorem \ref{mainthm}, then $F_{\dr_{0},\mu}(\varphi)$ is proper with respect to $J_{\dr_{0}}(\varphi),$ more precisely,
there exists an $\eta>0$ and a constant $C_{\eta}$ such that
\begin{equation}\label{eq:cma-proper}
 F_{\dr_{0},\mu}(\varphi)\geq\eta J_{\dr_{0}}(\varphi)-C_{\eta}.
\end{equation}
\end{theorem}
To prove this theorem, we will adapt the interpolation method using in \cite{LS} \cite{SW}. First, dut to the existence of the \ka-Einstein metric
$\dr_{KE}$ and no holomorphic vector field condition, we derive that $$F_{\dr_{0}}(\varphi)\geq\eta' J_{\dr_{0}}(\varphi)-C_{\eta'}$$
for some positive constants $\eta',C_{\eta'}$ by \cite{Ti97}, where $$F_{\dr_{0}}(\varphi)=J_{\dr_{0}}(\varphi)-\rev V\int_{M}\varphi\dr_{0}^{n}-\log\left(\rev V\int_{M}
e^{h_{0}-h-\varphi+a}\dr_{0}^{n}\right)$$
is the original Ding functional corresponding to smooth Monge-Ampere equation. To apply the interpolation, we also need to find a group of small data $\mu,\beta_{1},\cdots,\beta_{m}$ such that the corresponding properness holds.
That could be done by generalizing Berman's log $\alpha$-invariant estimate \cite{Ber}, which is the generalization of Tian's $\alpha$-invariant \cite{Ti87}. First we introduce the definition of log $\alpha$-invariant:
\begin{definition}
 Fix a smooth volume form $vol,$ for any \ka\ class $[\dr_{0}]$, we define log $\alpha$-invariant as below:
\begin{align*}
\alpha([\dr],D,\B)=&\sup\{\alpha>0: \exists C_{\alpha}<\infty\quad s.t. \quad\rev V\int_{M}\frac{e^{\alpha(\sup\phi-\phi)}vol}{\prod_{r=1}^{m}
||S_{r}||^{2(1-\beta_{r})}}\leq C_{\alpha}\\ &for\ any\ \phi\in PSH(M,\dr_{0})\}.
\end{align*}
\end{definition}
For the estimate of the log $\alpha$-invariant in case of simple normal crossing divisors, we have the following lemma:
\begin{lemma}\label{lem-log alpha}
For data $\mu,\beta_{r},\lambda_{r}$ where $1\leq r\leq n$ in the main theorem, we have
\begin{equation}\label{eq:log-alpha}
\alpha([\dr],D,\B)\geq\min\{\lambda_{1}\beta_{1},\cdots,\lambda_{m}\beta_{m},\alpha(K_{M}^{-1}),\alpha(K_{M}^{-1}|_{\cap_{r}D_{r}})\}.
\end{equation}
\end{lemma}
\begin{proof}
Similar to Berman's estimate \cite{Ber}, by Demailly's criterion  \cite{Dem2}, it suffices to prove the integrability in the definition for the functions $\frac{\log||\sigma_{k}||^{2}}{k}$ for each positive integer
$k$ where $\sigma_{k}\in H^{0}(kK_{M}^{-1}).$ For each $\sigma_{k}$ there exists nonnegative integers $l_{1},\cdots,l_{m}$ such that $$\sigma_{k}=\prod_{r=1}^{l}s_{r}^{\otimes l_{r}}\otimes\sigma',$$ where $\sigma'$ does not vanish identically on each irreducible component $D_{r}.$ For $t>0$ we have that
\begin{equation}\label{eq:log-alpha-expression}
\frac{e^{-t\frac{\log||\sigma_{k}||^{2}}{k}}}{\prod_{r=1}^{m}||S_{r}||^{2(1-\beta_{r})}}=
\frac{e^{-t\frac{\log||\sigma'||^{2}}{k}}}{\prod_{r=1}^{m}||S_{r}||^{2(1-\beta_{r}+\frac{t l_{r}}{k})}}.
\end{equation}
In case that $\sigma'\equiv 1,$ by the definition of $\lambda_{r}$ we have $\frac{l_{r}}{k}\geq\frac{1}{\lambda_{r}}$ thus for any $\delta>0$ if $t<\lambda_{r}\beta_{r}-\delta,$ the integral of \eqref{eq:log-alpha-expression} over $M$ is finite. On the other hand,
as $\sigma'$ does not vanish identically on each irreducible component $D_{r},$ the zero set of $\sigma'$
have at least complex codimension 1 on each $D_{r}.$ Meanwhile $deg\;\sigma'\leq k\;deg\;K_{M}^{-1},$ we have that
$\frac{\log||\sigma'||^{2}}{k}\in PSH(M,\dr_{0}).$ Similar to Berman's estimate, by Ohsawa-Takegoshi extension theorem we have that for each $\delta'>0$ and a small neighborhood $U$ of the divisor $D=\sum\limits_{r=1}^{m}D_{r}$ it holds that
$$\int_{U}\frac{e^{-t\frac{\log||\sigma'||^{2}}{k}}}{\prod_{r=1}^{m}||S_{r}||^{2(1-\delta')}}\leq C_{\delta'}\int_{\cap_{r}D_{r}}e^{-t\frac{\log||\sigma'||^{2}}{k}}.$$ Thus whenever $t\leq\min\{\lambda_{1}\beta_{1},\cdots,\lambda_{m}\beta_{m},\alpha(K_{M}^{-1}|_{\cap_{r}D_{r}})\}-\delta$ the integral over $U$ is finite.
On $M\setminus U$ the denominator is uniformly bounded so whenever $$t\leq\min\{\lambda_{1}\beta_{1},\cdots,\lambda_{m}\beta_{m},\alpha(K_{M}^{-1}),\alpha(K_{M}^{-1}|_{\cap_{r}D_{r}})\}-\delta$$ the integral is finite.
Combine these cases, the lemma is concluded.
\end{proof}
Given an $\alpha>0$ such that $$\rev V\int_{M}\frac{e^{\alpha(\sup\varphi-\varphi)+h_{0}-(1-\mu)h+a_{\mu}}\dr_{0}^{n}}{\prod_{r=1}^{m}
||S_{r}||^{2(1-\beta_{r})}}\leq C_{\alpha},$$ we could derive that
\begin{align*}
 \log C_{\alpha}&\geq\log\left(\rev V\int_{M}e^{\alpha(\sup\varphi-\varphi)+H_{0,\mu}-\log\frac{\dr_{\varphi}^{n}}{\dr_{0}^{n}}}\dr_{\varphi}^{n}\right)\\&\geq
 \frac{\alpha}{V}\int_{M}(\sup\varphi-\varphi)\dr_{\varphi}^{n}+\rev V\int_{M}H_{0,\mu}\dr_{\varphi}^{n}-\rev V\int_{M}\log\frac{\dr_{\varphi}^{n}}{\dr_{0}^{n}}\dr_{\varphi}^{n}
 \\&\geq\alpha I_{\dr_{0}}(\varphi)+\rev V\int_{M}H_{0,\mu}(\dr_{\varphi}^{n}-\dr_{0}^{n})+\rev V\int_{M}H_{0,\mu}\dr_{0}^{n}
 -\rev V\int_{M}\log\frac{\dr_{\varphi}^{n}}{\dr_{0}^{n}}\dr_{\varphi}^{n},
\end{align*}
then by the expression of th twisted Mabuchi functional, we have that
\begin{align*}
 \nu_{\dr_{0},\mu}(\varphi)=&\rev V\int_{M}\log\frac{\dr_{\varphi}^{n}}{\dr_{0}^{n}}\dr_{\varphi}^{n}+\rev V\int_{M}H_{0,\mu}(\dr_{0}^{n}-\dr_{\varphi}^{n})
 -\mu(I_{\dr_{0}}(\varphi)-J_{\dr_{0}}(\varphi))\\ \geq&\alpha I_{\dr_{0}}(\varphi)-\mu(I_{\dr_{0}}(\varphi)-J_{\dr_{0}}(\varphi))+C_{\alpha}'\\
 \geq&(\alpha-\frac{n}{n+1}\mu)I_{\dr_{0}}(\varphi)+C_{\alpha}'.
\end{align*}
Thus if we want to obtain the properness of the twisted Mabuchi functional we only need to make $\alpha-\frac{n}{n+1}\mu>0.$ By Lemma \ref{lem-log alpha} we want to
find $\mu>0$ such that $$\min\{\lambda_{1}\beta_{1},\cdots,\lambda_{m}\beta_{m},\alpha(K_{M}^{-1}),\alpha(K_{M}^{-1}|_{\cap_{r}D_{r}})\}>\frac{n}{n+1}\mu.$$
It is easy to see if $\mu$ is small enough we have that $\alpha(K_{M}^{-1}),\alpha(K_{M}^{-1}|_{\cap_{r}D_{r}})>\frac{n}{n+1}\mu.$ It remains to prove that for sufficiently
small $\mu>0$ and all $r,$ it holds that $\lambda_{r}\beta_{r}>\frac{n}{n+1}\mu,$ which is equivalent to
$$0<\lambda_{r}(1-(1-\mu)c_{r})-\frac{n}{n+1}\mu=\lambda_{r}(1-c_{r})+(\lambda_{r}c_{r}-\frac{n}{n+1})\mu,$$
thus we have that
\begin{proposition}\label{prop-log-alpha}
 In case that $c_{r}<1$ or $c_{r}=1$ with $\lambda_{r}>\frac{n}{n+1},$ for sufficiently small $\mu>0$ the corresponding twisted Mabuchi functional $\nu_{\dr_{0},\mu}(\varphi)$
 is proper.
\end{proposition}
Recall that in \cite{Li} \cite{Sh1} the properness of twisted Mabuchi functional implies the properness of the corresponding twisted Ding functional, which implies the
properness of the twisted Ding functional when $\mu>0$ sufficiently small. To derive the properness of all $\mu\in (0,1)$ we need such a lemma from \cite{LS} \cite{Sh1}:
\begin{lemma}
Suppose $0<\mu_{0}<\mu_{1},$ write $\mu=(1-t)\mu_{0}+t\mu_{1}$ where $0\leq t\leq 1,$ we have
$$\mu F_{\dr_{0},\mu}(\varphi)\geq(1-t)\mu_{0}F_{\dr_{0},\mu_{0}}(\varphi)+t\mu_{1}F_{\dr_{0},\mu_{1}}(\varphi).$$
\end{lemma}
The proof is an easy application of the concavity of logarithmic functions. By this lemma, we obtain the properness of the twisted Ding functional for all $\mu\in (0,1).$
The Theorem \ref{thm-cma-proper} is followed.

\section{Approximation procedure, uniform $C^{0}$-estimate and convergence}

In the last section, we obtained the properness of the twisted Ding functional for any data $\mu,\beta_{1},\cdots,\beta_{m}$ satisfying the
condition of Theorem \ref{mainthm}. Now we expect to establish the existence of corresponding conical \ka-Einstein metric. Recall that in
smooth case \cite{Ti97} and one smooth divisor case \cite{JMR} \cite{LS} \cite{SW}, they applied continuity method. One key ingredient is the
linear theory of Laplacian in ordinary and conical along one smooth divisor case. As such linear theory is still unknown in simple normal
crossing divisor case, in \cite{BBEGZ} they used pluripotential approach to obtain the existence of weak solution. Now we expect to use approximation approach
to establish the existence and regularity of such solutions. Recall that we consider the solution to the equation \eqref{eq:app-cke-cma}.
We want to establish uniform estimates for such solutions and prove that the solutions converge to the solution to the original conical Monge-
Ampere equation \eqref{eq:cke-cma}. Actually this approximation approach was first used by Tian \cite{Ti12} to approximate conical \ka-Einstein
metric and later in \cite{Sh1} it was used to approximate conic metrics with lower Ricci curvature bound. As \eqref{eq:app-cke-cma} is an ordinary smooth Monge-Ampere
equation, we could apply classical continuity method to solve this equation. More precisely, we consider the following equation:
\begin{equation}\label{eq:app-cke-cma-continuity}
 (\dr_{0}+\ddb\varphi_{\epsilon,t})^{n}=e^{h_{0}-(1-\mu)h-t\varphi_{\epsilon,t}-\sum_{r=1}^{m}(1-\beta_{r})
 \log(||S_{r}||^{2}+\epsilon)+a_{\mu,\epsilon}}\dr_{0}^{n},
\end{equation}
for $t\in [0,\mu].$ Suppose $E\in [0,\mu]$ is the solvable set for $t.$ First, when $t=0$ by Yau's solution to Calabi conjecture
\cite{Yau}, it is solvable so $0\in E.$ Next, we know the linear operator at some $t\in E$ is $\Delta_{t}+t,$ where $\Delta_{t}$ is
the Laplacian with respect to the metric $\dr_{\epsilon,t}:=\dr_{0}+\ddb\varphi_{\epsilon,t}.$ Thus we have the following inequality for the Ricci curvature of  $\dr_{\epsilon,t}:$
\begin{align*}
Ric(\dr_{\epsilon,t})=&\dr_{0}+(1-\mu)\ddb h+t\ddb\varphi_{\epsilon,t}+\sum_{r=1}^{r}(1-\beta_{r})\ddb\log(||S_{r}||^{2}+\epsilon)
\\=&\mu\dr_{0}+t\ddb\varphi_{\epsilon,t}+\sum_{r=1}^{r}(1-\beta_{r})(\ddb\log(||S_{r}||^{2}+\epsilon)+R(||\cdot||_{r}))\\=&
t\dr_{\epsilon,t}+(\mu-t)\dr_{0}+\sum_{r=1}^{r}(1-\beta_{r})(\frac{\epsilon R(||\cdot||_{r})}{||S_{r}||^{2}+\epsilon}+\frac{\epsilon
\sqrt{-1}DS_{r}\wedge\overline{DS_{r}}}{(||S_{r}||^{2}+\epsilon)^{2}})\\>&t\dr_{\epsilon,t},
\end{align*}
where the inequality holds as all $D_{r}$ are semi-ample. By Bochner's formula, we know that the linear operator $\Delta_{t}+t$ is invertible thus $E$ is open.

The most difficult part is the closeness, i.e., the uniform $C^{0}$-estimate for $\varphi_{\epsilon,t}$ for each $t>0.$ To establish such estimate, we will prove the properness of corresponding Ding functionals and obtain a uniform upper bound for them. Thus we can give a uniform bound for
$J_{\dr_{0}}(\varphi_{\epsilon,t})$ which implies the uniform $C^{0}$-bound for $\varphi_{\epsilon,t}.$

We set $$H_{0,\mu,\epsilon}=h_{0}-(1-\mu)h-\sum_{r=1}^{m}(1-\beta_{r})\log(||S_{r}||_{r}^{2}+\epsilon)+a_{\mu,\epsilon},$$ then by the choice of $a_{\mu,\epsilon}$
we know that $\int_{M}(e^{H_{0,\mu,\epsilon}}-1)\dr_{0}^{n}=0.$ Now we define the approximating Ding functional as following:
\begin{equation}\label{eq:app-ding}
 F_{\mu,\epsilon}(\varphi):=J_{\dr_{0}}(\varphi)-\rev V\int_{M}\varphi\dr_{0}^{n}-\rev\mu\log(\rev V\int_{M}e^{H_{0,\mu,\epsilon}-\mu\varphi}\dr_{0}^{n}),
\end{equation}
which is the Lagrangian of the approximating equation \eqref{eq:app-cke-cma}. Recall the computation in \cite{Ti97}, suppose $\varphi_{\epsilon,t}$ solves the equation
\eqref{eq:app-cke-cma-continuity}, take the derivative of both sides of
\eqref{eq:app-cke-cma-continuity} we have that
$$\Delta_{\epsilon,t}\dot{\varphi}_{\epsilon,t}=-\varphi_{\epsilon,t}-t\dot{\varphi}_{\epsilon,t},$$
where $\Delta_{\epsilon,t}$ is the Laplacian with respect to $\dr_{\epsilon,t}=\dr_{0}+\ddb\varphi_{\epsilon,t}.$
Consequently,
\begin{align*}
 \frac{d}{dt}(t(J_{\dr_{0}}(\varphi_{\epsilon,t})-\rev V\int_{M}\varphi_{\epsilon,t}\dr_{0}^{n}))
 &=J_{\dr_{0}}(\varphi_{\epsilon,t})-\rev V\int_{M}\varphi_{\epsilon,t}\dr_{0}^{n}-\frac{t}{V}\int_{M}\dot{\varphi}_{\epsilon,t}\dr_{\epsilon,t}^{n}\\
 &=J_{\dr_{0}}(\varphi_{\epsilon,t})-\rev V\int_{M}\varphi_{\epsilon,t}(\dr_{0}^{n}-\dr_{\epsilon,t}^{n})\\&=
 -(I_{\dr_{0}}(\varphi_{\epsilon,t})-J_{\dr_{0}}(\varphi_{\epsilon,t}))\leq 0.
\end{align*}
Integrate from 0 to $t$ and make use of the concavity of logarithmic function we deduce that
\begin{align*}
 F_{\mu,\epsilon}(\varphi_{\epsilon,t})&\leq-\rev\mu\log(\rev V\int_{M}e^{H_{0,\mu,\epsilon}-\mu\varphi_{\epsilon,t}}\dr_{0}^{n})
 \\&=-\rev\mu\log(\rev V\int_{M}e^{H_{0,\mu,\epsilon}-t\varphi_{\epsilon,t}-(\mu-t)\varphi_{\epsilon,t}}\dr_{0}^{n})
 \\&\leq-\frac{\mu-t}{\mu}\rev V\int_{M}\varphi_{\epsilon,t}\dr_{\epsilon,t}^{n}.
\end{align*}
As for each $t\geq t_{0}>0,$ it holds that $Ric(\dr_{t,\epsilon})>t\dr_{t,\epsilon}\geq t_{0}\dr_{t,\epsilon}$
and the volume is fixed, we could have a uniform control of the Sobolev constant and the first eigenvalue.
Thus it follows from standard Moser's iteration that
$$-\inf_{M}\varphi_{\epsilon,t}\leq C(C'+\rev V\int_{M}\varphi_{\epsilon,t}\dr_{\epsilon,t}^{n}),$$ where all constants depends on $t_{0}.$ As $\inf_{M}\varphi_{\epsilon,t}$ by the normalization condition we obtain that
 \begin{equation}\label{eq:app-ding-upper}
 F_{\mu,\epsilon}(\varphi_{\epsilon,t})\leq C.
 \end{equation}
 As $n+\Delta_{0}\varphi_{\epsilon,t}=tr_{\dr_{0}}\dr_{\epsilon,t}>0,$ by standard Green formula it holds
that $$\sup_{M}\varphi_{\epsilon,t}\leq c+\rev V\int_{M}\varphi_{\epsilon,t}\dr_{0}^{n}.$$
Combine these two estimates for $\varphi_{\epsilon,t}$ we have that
\begin{equation}\label{eq:osc}
osc_{M}\varphi_{\epsilon,t}=\sup_{M}\varphi_{\epsilon,t}-\inf_{M}\varphi_{\epsilon,t}\leq C(1+I_{\dr_{0}}(\varphi_{\epsilon,t}))\leq (n+1)C(1+J_{\dr_{0}}(\varphi_{\epsilon,t})).
\end{equation}
Now considering that
\begin{align*}
H_{0,\mu,\epsilon}&=h_{0}-(1-\mu)h-\sum_{r=1}^{m}(1-\beta_{r})\log(||S_{r}||_{r}^{2}+\epsilon)
+a_{\mu,\epsilon}\\&\leq h_{0}-(1-\mu)h-\sum_{r=1}^{m}(1-\beta_{r})\log||S_{r}||_{r}^{2}+a_{\mu,\epsilon}
=H_{0,\mu}-a_{\mu}+a_{\mu,\epsilon},
\end{align*}
as $a_{\mu,\epsilon}$ and $a_{\mu}$ are uniformly bounded, it is easy to obtain the properness of
the approximating Ding functional from Theorem \ref{thm-cma-proper} :
\begin{equation}\label{eq:app-ding-proper}
 F_{\mu,\epsilon}(\varphi)\geq\eta J_{\dr_{0}}(\varphi)-C'_{\eta}.
\end{equation}
Combine \eqref{eq:app-ding-upper} \eqref{eq:osc} \eqref{eq:app-ding-proper} and note that by normalization
condition $$\sup_{M}\varphi_{\epsilon,t}\geq 0,\;\inf_{M}\varphi_{\epsilon,t}\leq 0,$$ we obtain that
$|\varphi_{\epsilon,t}|\leq C(t_{0})$ for $t\geq t_{0}>0.$ Thus the solvable set $E$ is close for each
$\epsilon>0.$ To summarize, we have that
\begin{theorem}\label{thm-uniform C0}
For each $\epsilon>0$ there exists a unique smooth solution $\varphi_{\epsilon}$ to \eqref{eq:app-cke-cma}
such that $\dr_{\epsilon}=\dr_{0}+\ddb\varphi_{\epsilon}$ have uniform lower Ricci curvature bound $\mu.$ Moreover, $|\varphi_{\epsilon}|\leq C$ which is independent of $\epsilon.$
\end{theorem}

By the uniform $C^{0}$-estimate for the approximating solutions $\varphi_{\epsilon}$ for any
$\epsilon>0,$ we can prove that there exists one subsequence converging to a weak solution
of the conical Monge-Ampere equation \eqref{eq:cke-cma} and such weak solution is unique.

First, by Kolodziej's H\"{o}ldr estimate \cite{K}, we derive that $||\varphi_{\epsilon}||_{C^{\alpha}}$ is uniformly bounded for some $\alpha\in (0,1)$ from Theorem \ref{thm-uniform C0}. Thus there exists a subsequence of $\varphi_{\epsilon}$ which converge to a function
$\varphi\in PSH(M,\dr_{0})\cap L^{\infty}(M)$ such that $\varphi$ is a weak solution to the equation \eqref{eq:cke-cma} in distribution sense. By Berndtsson's uniqueness theorem \cite{Bern}, the weak solution to \eqref{eq:cke-cma} is unique. Moreover, by Chern-Lu's Inequality \cite{Lu}, as $Ric(\dr_{\epsilon})\geq\mu\dr_{\epsilon},$ we have that
$$\Delta_{\epsilon}\log tr_{\dr_{\epsilon}}\dr_{0}\geq -a\;tr_{\dr_{\epsilon}}\dr_{0},$$
where $\Delta_{\epsilon}$ is the Laplacian of $\dr_{\epsilon}$ and $a$ is the upper bound of the bisectional
curvature of $\dr_{0}.$ Then put $$u=\log tr_{\dr_{\epsilon}}\dr_{0}-(a+1)\varphi_{\epsilon}$$ then by the
inequality above we obtain that $$\Delta_{\epsilon}u\leq e^{u-(a+1)c}-n(a+1),$$ which implies $u\leq C$ by
maximal principle, so we have $c_{1}\dr_{0}\leq\dr_{\epsilon}.$ By the equation \eqref{eq:cke-cma},
we obtain that $$c_{1}\dr_{0}\leq\dr_{\epsilon}\leq\frac{c_{2}\dr_{0}}
{\prod_{r=1}^{m}(||S_{r}||^{2}+\epsilon)^{(1-\beta_{r})}}.$$ By standard high order estimate, we have that
$$||\varphi_{\epsilon}||_{C^{l}(K)}\leq C(l,K)$$ uniformly on each compact set $K\in M\setminus D,$ which implies
\begin{proposition}\label{prop-local convergence}
$\dr_{\epsilon}$ converges to $\dr_{\varphi}=\dr_{0}+\ddb\varphi$ in current sense, and the convergence is $C^{\infty}$ on each compact set $K\in M\setminus D.$ Moreover $\dr_{\varphi}$ is smooth outside the divisor $D.$
\end{proposition}

\section{Higher regularities}

Assuming the existence of $L^{\infty}$ weak solution to \eqref{eq:cke-cma}, the next step is to establish the
Laplacian estimate, i.e., the equivalence of the metric $\dr_{\varphi}$ and the standard conic metric
$$\dr_{bg}:=\dr_{0}+\epsilon\sum_{r=1}^{m}\ddb||S_{r}||^{2\beta_{r}}.$$
In one irreducible divisor case, Brendle proved the $C^{3}$-estimates in \cite{Br} using Calabi's estimate when $\beta<\frac{1}{2}.$ For any $\beta<1,$ in \cite{JMR} the Laplacian estimate follows from
Chern-Lu's Inequality based on Li-Rubinstein's upper bisectional curvature bound estimate for standard type conic metrics. In simple normal crossing divisor case, one way is to set up a new approximating approach \cite{CGP} \cite{GP}. However as the approximating conic metrics have no uniform control of bisectional curvature from either above or below, they need to establish a much more complicated Laplacian estimate based on Paun's trick. In this paper we will adapt  Datar-Song's simpler estimate \cite{DS} which based on Li-Rubinstein's estimate to obtain the Laplacian estimate. The basic idea is that for each $r$ keep the divisor $D_{r}$ component and regularize other components, then compare each approximating solution metric with standard conic metrics along $D_{r}$ with cone angle $2\pi\beta_{r}$ by Li-Rubinstein's curvature estimate and finally the Laplacian estimates follow from these. And we will discuss another possible approach to the Laplacian estimates based on our new curvature estimate after the proof of Theorem \ref{thm-upper}.\\

In Fano case as \cite{DS} \cite{GP}, for the solution $\varphi\in PSH(\dr_{0})$ to \eqref{eq:cke-cma}, we use Demailly's approximation to construct a smooth sequence $\phi_{j}\in PSH(\dr_{0})$ such that $\phi_{j}\searrow\varphi.$ Now consider the approximating equation of \eqref{eq:cke-cma}:
\begin{equation}\label{eq:cke-cma-app-1}
(\dr_{0}+\ddb\varphi_{j})^{n}=e^{h_{0}-(1-\mu)h-\mu\phi_{j}-\sum_{r=1}^{m}(1-\beta_{r})\log||S_{r}||^{2}+a_{\mu,j}}\dr_{0}^{n},
\end{equation}
where $a_{\mu,j}$ is normalization constants and converges to $a_{\mu}.$ It follows from Kolodziej's works \cite{Ko} \cite{K} that $|\varphi_{j}|\leq C$ uniformly and $|\varphi_{j}-\varphi|_{C^{0}}\to 0.$ Thus the Laplacian estimate for $\varphi$ will follow from uniform Laplacian estimate for $\varphi_{j}.$
In the following for simplicity we drop the index $j$ from \eqref{eq:cke-cma-app-1}. As \cite{DS}, denote $$f_{i}:=\sum_{r\neq i}^{m}(1-\beta_{r})\log||S_{r}||^{2}-(h_{0}-(1-\mu)h-\mu\phi+a_{\mu}).$$ Then as there exists large enough $A$ such that $A\dr_{0}+\ddb f_{i}>0,$ by
Demailly's approximation, we can find a smooth sequence of functions $F_{i,k}\searrow f_{i}$ and consider the following equation:
\begin{equation}\label{eq:cke-cma-app-2}
\dr_{i,k}^{n}:=(\dr_{0}+\ddb\varphi_{i,k})^{n}=\frac{e^{-F_{i,k}+c_{i,k}}\dr_{0}^{n}}{||S_{i}||^{2(1-\beta_{i})}},
\end{equation}
with $\sup\varphi_{i,k}=0,$ where $c_{i,k}$ is normalization constants. By \cite{JMR} there exists $\varphi_{i,k}\in C^{2,\alpha,\beta_{i}}$ for some $\alpha\in (0,1).$ Integrate both sides it follows that $c_{i,k}$ are uniformly bounded. As the right side of \eqref{eq:cke-cma-app-2} is uniformly bounded in $L^{1+\epsilon}$ for some $\epsilon>0,$ by Kolodziej's estimate \cite{Ko} $|\varphi_{i,k}|_{C^{0}}$ are uniformly bounded. Moreover as $\frac{e^{-F_{i,k}+c_{i,k}}\dr_{0}^{n}}{||S_{i}||^{2(1-\beta_{i})}}\to \frac{e^{-f_{i}}\dr_{0}^{n}}{||S_{i}||^{2(1-\beta_{i})}}$ uniformly in $L^{1}$ by Kolodziej's stability \cite{K} it follows that $\varphi_{i,k}\to\varphi$ uniformly in $C^{0}$ for each $i$ and $k\to\infty.$ \\

Now as $\varphi_{i,k}\in C^{2,\alpha,\beta_{i}},$ write the standard conic metric $\dr_{bg,i}:=\dr_{0}+\epsilon\ddb||S_{i}||^{2\beta_{i}}$ with angle $2\pi\beta_{i}$
along $D_{i},$ it follows that $tr_{\dr_{i,k}}\dr_{bg,i}$ is bounded. Then consider $$Q_{i,k}:=\log(||S_{i}||^{2\delta}tr_{\dr_{i,k}}\dr_{bg,i})-
A(\varphi_{i,k}-\epsilon||S_{i}||^{2\beta_{i}}),$$ where $\delta,A>0,$ as $Ric(\dr_{i,k})\geq-C\dr_{bg,i}$ by \eqref{eq:cke-cma-app-2} for uniform constant $C,$ and the bisectional curvature of $\dr_{bg,i}$ is bounded above by $C'$ on $M\setminus D_{i}$ by Li-Rubinstein's estimate \cite{JMR}, it follows from Chern-Lu's Inequality that
$$\Delta Q_{i,k}\geq (A-C')tr_{\dr_{i,k}}\dr_{bg,i}+\delta tr_{\dr_{i,k}}R(||\cdot||_{i})-An.$$ As $tr_{\dr_{i,k}}\dr_{bg,i}$ is bounded and $\delta>0,$ $Q$ tends to $-\infty$ near $D_{i}$ which implies that the maximal
value of $Q_{i,k}$ is attained at some point $p\in M\setminus D_{i}.$ Thus by maximal principle, at the maximal point of $Q_{i,k}$ it holds that
$$(A-C')tr_{\dr_{i,k}}\dr_{bg,i}+\delta tr_{\dr_{i,k}}R(||\cdot||_{i})-An\leq 0.$$ As all $R(||\cdot||_{i})\geq 0,$ by taking $A\geq C'+1$ we have
$tr_{\dr_{i,k}}\dr_{bg,i}\leq C''$ for some uniform constant $C''$ at that point which is independent of $k.$ Thus $Q_{i,k}\leq C'''$ for some uniform constant $C'''$ and it follows that
$$tr_{\dr_{i,k}}\dr_{bg,i}\leq\frac{1}{c||S_{i}||^{2\delta}}$$ for some uniform constant $c>0.$ Let $\delta$ tend to 0, we will have that $\dr_{i,k}\geq c\dr_{bg,i}.$ Take the limit as $k\to\infty$ it holds that $\dr_{\varphi}=\dr_{0}+\ddb\varphi\geq c\dr_{bg,i}$ in current sense. Thus it follows that
$\dr_{\varphi}\geq c'\dr_{bg}.$ Considering the Monge-Ampere equation \eqref{eq:cke-cma} and letting $\phi_{i}\searrow\varphi,$ we will have
\begin{proposition}\label{prop-laplacian}
There exist $C_{1},C_{2}>0$ such that $C_{1}\dr_{bg}\leq\dr_{\varphi}\leq C_{2}\dr_{bg}.$
\end{proposition}

  Finally, to finish the proof of Theorem \ref{mainthm} we need to establish a $C^{2,\alpha,\B}$-estimate for the
solution $\varphi,$ i.e., to show $\ddb\varphi$ is $C^{\alpha}$ with respect to the standard conic metric $\dr_{cone}.$ By \cite{Br} in one smooth divisor case if $\beta<\frac{1}{2}$ we can even show that $||\varphi||_{C^{3}}$ is bounded using Calabi's 3rd order estimate, because in this case the bisectional curvature is uniformly bounded outside the divisor. Actually this result could be generalized to
simple normal crossing divisors whose cone angles are all smaller than $\frac{1}{2}$ as the bisectional curvature is also uniformly bounded outside the divisor by the proof of Theorem \ref{thm-upper} in the next section. However, in case that $\beta\geq\frac{1}{2}$ one can only have the$C^{\alpha}$-bound for $\ddb\varphi$ where $\alpha\in(0,\min\{\frac{1}{\beta}-1,1\}).$ Based on Tian's 3rd order estimate \cite{Ti14}, we could establish similar estimates in general simple normal crossing case. Most of the details are parallel to \cite{Ti14} which works for one divisor case. We will sketch the main steps to establish such estimate. We refer \cite{Sh3} for more details in case of simple normal crossing divisor case when some angle $\beta_{r}>\frac{1}{2}$ as the conical \ka-Einstein metric can also be considered as the stationary case of the conical \ka-Ricci flow in \cite{Sh3}:
\begin{proposition}\label{prop-C3}
 If the solution $\varphi$ of conical Monge-Ampere equation \eqref{eq:cke-cma} satisfies that $C_{1}\dr_{bg}\leq\dr_{\varphi}\leq C_{2}\dr_{bg},$ then for
 any point $p\in M,\;r<1$ and $\alpha\in(0,\min\{1,\rev{\beta_{1}}-1,\cdots,\rev{\beta_{m}}-1\})$ there exists a constant $C_{\alpha}>0$ such that
 $$\int_{B_{r}(p)}|\partial\bar{\partial}\partial\varphi|^{2}\dr_{cone}^{n}\leq C_{\alpha}r^{2n-2+2\alpha}.$$
 Moreover, $\ddb\varphi$ is $C^{\alpha}$-bounded.
\end{proposition}
Given this proposition the proof of Theorem \ref{mainthm} is complete. In the following we sketch the main steps. For simplicity we deduce the equation \eqref{eq:cke-cma} to the form
 \begin{equation}\label{eq:cma-holder}
 det(u_{i\bar{j}})=\frac{e^{F}}{\prod_{i=1}^{m}||S_{i}||^{2(1-\beta_{i})}}
 \end{equation}
where $\ddb F$ is uniformly bounded with respect to $\dr_{bg}$ by the above Laplacian estimates. Consider the ball $\mathbb{B}_{\B}(r)$ in $\mathbb{C}^{m}\times\mathbb{C}^{n-m}$ with coordinates $(w_{1},\cdots,w_{m},w')$ satisfying $$\sum_{i=1}^{m}|w_{i}|^{2}+|w'|^{2}\leq r^{2}\;and\;w_{i}=\rho_{i}e^{\sqrt{-1}\theta_{i}}\;for\;\rho_{i}\in[0,r],\theta_{i}\in[0,2\pi\beta_{i}].$$ Actually this ball is the geometric model
of the conic region near the simple normal crossings and there exists a natural map from $\mathbb{B}_{\B}(r)$ to a region in $\C$ that is defined by
$w_{1}=\rev{\beta_{1}}z_{1}^{\beta_{1}},\cdots,w_{m}=\rev{\beta_{m}}z_{m}^{\beta_{m}},w_{m+1}=z_{m+1},\cdots,w_{n}=z_{n}.$ Under these coordinates we can
redefine the standard conic metric in \eqref{eq:def-conic} as $\dr_{\B}:=\sqrt{-1}\sum\limits_{i=1}^{m}dw_{i}\wedge d\bar{w}_{i}.$ Basically all the derivatives
of $u$ are taken with respect to such coordinates and we denote $U^{ij}:=det(u_{k\bar{l}})u^{i\bar{j}},$ then we have following elementary facts:
\begin{lemma}\label{lem-fact}(Lemma 2.1-2.2 \cite{Ti14})
(1)For each $j,$ it holds that $U^{i\bar{j}}_{i}=0;$ (2) For any positive $\lambda_{1},\cdots\lambda_{n},$ it holds that
$$|\frac{\lambda}{\lambda_{i}}u_{i\bar{i}}-det(u_{p\bar{q}})-(n-1)\lambda|\leq C\sum_{i,j}|u_{i\bar{j}}-\lambda_{i}\delta_{ij}|^{2}.$$
\end{lemma}
Next we could generalize the Sobolev Inequality for one divisor case (Lemma 2.3 \cite{Ti14}) to the simple normal crossing case:
\begin{lemma}\label{lem-sobolev}
There is a constant $C_{\B}$ which depends on $\beta_{1},\cdots\beta_{m}$ such that for any smooth function $h$ on $\mathbb{B}_{\B}=\mathbb{B}_{\B}(1)$ with boundary condition
$$h(w_{1},\cdots,\rho_{i}e^{\sqrt{-1}2\pi\beta_{i}},\cdots,w_{n})=e^{\sqrt{-1}2\pi(1-\beta_{i})}h(w_{1},\cdots,\rho_{i},\cdots,w_{n}),$$
we have $$\int_{\mathbb{B}_{\B}}|h|^{2}\dr_{\B}^{n}\leq C_{\B}(\int_{\mathbb{B}_{\B}}|dh|^{\frac{2n}{n+1}}\dr_{\B}^{n})^{\frac{n+1}{n}}.$$ Note that $C_{\B}$ blows up when one $\beta_{i}$ tends to 1.
\end{lemma}
By above two lemmas and Gehring's inverse H\"older Inequality we have the following estimate:
\begin{lemma}\label{lem-gehring}(Lemma 2.4 \cite{Ti14})
There is some $q>2,$ which may depend on $\beta_{1},\cdots\beta_{m},||u_{i\bar{j}}||_{L^{\infty}}$ and $||F_{i\bar{j}}||_{L^{\infty}}$ such that for any $B_{2r}(y)\subset U,$ we have
$$\left(r^{-2n}\int_{B_{r}(y)}(1+|\nabla\dr|^{2})^{\frac{q}{2}}\dr_{\B}^{n}\right)^{\frac{2}{q}}\leq Cr^{-2n}\int_{B_{r}(y)}(1+|\nabla\dr|^{2})\dr_{\B}^{n},$$
where $\dr=\ddb u$ and $C$ is a uniform constant.
\end{lemma}
Compare $\ddb u$ with the solution to a boundary value problem, consider the above lemma and use a monotonicity property we have
\begin{lemma}\label{lem-compare}(Lemma 2.5 \cite{Ti14})
For any $y\in V$ and $B_{4r}(y)\subset U$ and $\sigma<r,$ we have
\begin{align}\label{eq:compare}
&\int_{B_{\sigma}(y)}|\nabla\dr|^{2}\dr_{\B}^{n}-cr^{2n}\nonumber\\
\leq&C\left[\left(\frac{\sigma}{r}\right)^{2n-4+2\min\{\beta_{1}^{-1},\cdots,\beta_{m}^{-1}\}}+\left(r^{2-2n}\int_{B_{r}(y)}|\nabla\dr|^{2}
\dr_{\B}^{n}\right)^{\frac{q-2}{q}}\right]\int_{B_{r}(y)}|\nabla\dr|^{2}\dr_{\B}^{n}.
\end{align}
\end{lemma}
Now we still need the following lemma which can be proved by integration by parts and Moser's iteration:
\begin{lemma}\label{lem-small}(Lemma 2.7 \cite{Ti14})
For any $\epsilon_{0}>0,$ there is an $l$ depending on $\epsilon_{0},$ $||\Delta u||_{L^{\infty}}$ and $\inf\Delta F$ satisfying that for any $\tilde{r}>0$ with $B_{\tilde{r}}(y)\subset U,$ there is $r\in[2^{-l}\tilde{r},\tilde{r}],$ such that
\begin{equation}\label{eq:small}
r^{2-2n}\int_{B_{r}(y)}|\nabla\dr|^{2}\dr_{\B}^{n}\leq\epsilon_{0}.
\end{equation}
\end{lemma}
Combine above two lemmas we could have the following monotonicity inequality:
$$\sigma^{2\nu}+\sigma^{2-2n}\int_{B_{\sigma}(y)}|\nabla\dr|^{2}\dr_{\B}^{n}\leq\xi
\left(r^{2\nu}+r^{2-2n}\int_{B_{r}(y)}|\nabla\dr|^{2}\dr_{\B}^{n}\right),$$
where $r\leq r_{0}$ small, $\sigma=\lambda r$ for some $\lambda\in(0,1),$ $\nu=\min\{\beta_{1}^{-1},\cdots,\beta_{m}^{-1}\}-1>\alpha$ and $\xi=\lambda^{\alpha}.$ Then Proposition \ref{prop-C3} follows from a standard iteration and the whole proof of Theorem \ref{mainthm} is complete.
\section{Proof of Theorem \ref{thm-upper}}

In this section we supply a detailed proof for Theorem \ref{thm-upper}, which is not only a uniform estimate for the conic background geometry but also provide another possible approach for the Laplacian estimates in conical geometric problems. We will compute along Li-Rubinstein's approach in the appendix of \cite{JMR}.
First, as \cite{JMR}, we denote $\hat{g},g$ as the \ka\ metrics associated to $\dr_{0},\dr.$ Without loss of confusion we temporarily still use $\dr_{0}$ to represent the
whole background metric even if the precise form is $\dr_{0}+\ddb\varphi$ in the second case of the main theorem. Here we only deal with the most complicated case
that the point we consider is near the intersection of all divisors. For other cases the estimates will be similar but easier. Now we extend Lemma A.2 in \cite{JMR}
which came from \cite{TY} originally to choose appropriate local holomorphic frames and coordinate system near the intersection (actually by the private communication
Y. Rubinstein also gave the same extension), which we will use in the following computation:

\begin{lemma}\label{lem-coordinate}
 There exists $\epsilon_{0}>0$ such that if $0\leq dist_{\hat{g}}(p,D_{r})\leq\epsilon_{0}$ for all $r=1,\cdots,m,$ then we can choose local holomorphic
frames $e_{r}$ of each holomorphic line bundle $[D_{r}]$ and local holomorphic coordinates $(z_{1},\cdots,z_{n})$ valid in a neighborhood of p, such that
(i)$S_{r}=z_{r}e_{r},$ and $a_{r}:=||e_{r}||_{r}^{2}$ satisfies $a_{r}(p)=1,\;da_{r}(p)=0,\;\frac{\partial^{2}a_{r}(p)}{\partial z_{i}\partial z_{j}}=0,$
and (ii) $\hat{g}_{i\bar{j}}(p)=0$ for $i\leq m,j>m,$ and $\hat{g}_{i\bar{j},k}(p)=\frac{\partial}{\partial z_{k}}\dr_{0}(\frac{\partial}{\partial z_{i}},
\frac{\partial}{\partial\bar{z}_{j}})|_{p}=0,$ whenever $j>m.$
\end{lemma}
\begin{proof}
Actually this proof can be modified from the proof of Lemma A.2 in \cite{JMR}.  Fix any point q on the intersection of $D_{1},\cdots,D_{m},$ and choose
local holomorphic frames $e'_{1},\cdots,e'_{m}$ and holomorphic coordinates $(w_{1},\cdots,w_{n})$ in $B_{\hat{g}}(q,\epsilon(q))$ for some sufficient small
$\epsilon(q).$ Let $S_{r}=f'_{r}e'_{r}$ with $f'_{r}$ holomorphic functions and $||e'_{r}||_{r}^{2}=c_{r}.$ Let $e_{r}=F_{r}e'_{r}$ for some nonvanishing
holomorphic functions $F_{r}$ to be determined later. Then we have $a_{r}=||F_{r}e'_{r}||_{r}^{2}=|F_{r}|^{2}c_{r}.$ Now fix any point $p\in B_{\hat{g}}
(q,\epsilon(q)).$ In order for $a_{r}$ to satisfy condition (i) with respect to
coordinates $(w_{1},\cdots,w_{n})$ at point p, we choose $F_{r}$ such that $F_{r}(p)=c_{r}(p)^{-\frac{1}{2}}$ and
\begin{align*}
\partial_{w_{i}}F_{r}(p)&=-c_{r}^{-1}F_{r}(p)\partial_{w_{i}}c_{r}(p)=-c_{r}^{-\frac{3}{2}}\partial_{w_{i}}c_{r}(p)\\
\partial_{w_{i}}\partial_{w_{j}}F_{r}(p)&=-c_{r}^{-1}(F_{r}\partial_{w_{i}}\partial_{w_{j}}c_{r}+\partial_{w_{i}}c_{r}
\partial_{w_{j}}F_{r}+\partial_{w_{i}}F_{r}\partial_{w_{j}}c_{r})(p)\\&=-c_{r}^{-\frac{3}{2}}\partial_{w_{i}}
\partial_{w_{j}}c_{r}(p)+2c_{r}^{-\frac{5}{2}}\partial_{w_{i}}c_{r}\partial_{w_{j}}c_{r}(p).
\end{align*}
Since $c_{r}=||e'_{r}||_{r}^{2}$ is nonzero as $\epsilon(q)$ is small, which implies $|w-w(p)|$ is small, we can assume $F_{r}\neq 0$ in $B_{\hat{g}}
(q,\epsilon(q)).$ Now $S_{r}=f'_{r}e'_{r}=f_{r}e_{r}$ with $f_{r}=f'_{r}F_{r}^{-1}$ holomorphic functions. As $D_{r}=\{z_{r}=0\}$ are smooth divisors,
we can assume that $\partial_{w_{r}}f_{r}(q)\neq 0,$ but $\partial_{w_{s}}f_{r}(q)=0$ for $s\neq r$ among $1,\cdots,m,$ and choose small $\epsilon(q),$
we can assume that $\partial_{w_{r}}f_{r}\neq 0$ but $\partial_{w_{s}}f_{r}(q)$ sufficiently small for $s\neq r$ among $1,\cdots,m,$ in $B_{\hat{g}}
(q,\epsilon(q)).$ Then by the inverse function theorem,
$$z_{1}=f_{1}(w_{1},\cdots,w_{n}),\cdots,z_{m}=f_{m}(w_{1},\cdots,w_{n}),z_{m+1}=w_{m+1},\cdots,z_{n}=w_{n}$$ are holomorphic coordinates in
$B_{\hat{g}}(q,\epsilon(q)/2)$ and now $S_{r}=f_{r}(w)e_{r}=z_{r}e_{r}.$ By the chain rule,
(i) holds locally in the new coordinates. By covering argument (i) holds globally.\\

For (ii), we denote by $w^{1},\cdots,w^{n}$ the coordinates obtained in (i). First we can make a coordinate change such that at the point p,
$\frac{\partial}{\partial z_{j}}_{j>m}$ perpendicular to $\frac{\partial}{\partial z_{j}}_{j\leq m}.$ Let
$$\tilde{z}^{k}=w^{k}-w^{k}(p)+\frac{1}{2}b_{st}^{k}(w^{s}-w^{s}(p))(w^{t}-w^{t}(p)),$$ with
$b_{st}^{k}=b_{ts}^{k},$ define a new coordinate system. Then we have that
\begin{align*}
\dr_{0}&(\frac{\partial}{\partial w^{i}},\frac{\partial}{\partial\bar{w}^{j}})=\dr_{0}(\frac{\partial}{\partial \tilde{z}^{i}},\frac{\partial}
{\partial\bar{\tilde{z}}^{j}})+\hat{g}_{t\bar{j}}b_{is}^{t}(w^{s}-w^{s}(p))+\hat{g}_{i\bar{t}}\overline{b_{sj}^{t}
(w^{s}-w^{s}(p))}\\+&O(\sum_{r=1}^{n}|w^{r}-w^{r}(p)|^{2}),\\
d_{i\bar{j}k}&:=\frac{\partial}{\partial w^{k}}\dr_{0}(\frac{\partial}{\partial w^{i}},\frac{\partial}{\partial\bar{w}^{j}})|_{p}
=\frac{\partial}{\partial\tilde{z}^{k}}\dr_{0}(\frac{\partial}
{\partial\tilde{z}^{i}},\frac{\partial}{\partial\bar{\tilde{z}}^{j}})+
\hat{g}_{t\bar{j}}(p)b_{ik}^{t}=:e_{i\bar{j}k}+\hat{g}_{t\bar{j}}(p)b_{ik}^{t}.
\end{align*}
Let $\hat{g}'_{r\bar{s}}:=\hat{g}_{r\bar{s}},$ for $r,s>m,$ and denote the inverse of the $(n-m)\times(n-m)$ matrix $[\hat{g}'_{r\bar{s}}]$
by $[\hat{g}'^{r\bar{s}}].$ Let $b_{ik}^{r}=0$ for $r=1,\cdots,m.$ Then for each $j>m,$ the equations can be rewritten as
$d_{i\bar{j}k}-\sum\limits_{t>m}\hat{g}_{t\bar{j}}(p)b_{ik}^{t}=e_{i\bar{j}k}.$ Hence if we define $b_{ik}^{t}=\sum\limits_{j>m}\hat{g}'^{t\bar{j}}
d_{i\bar{j}k}$ for each $t>m,$ we will have that $e_{i\bar{j}k}=0$ for each $j>m.$ Finally, we set $z^{r}=\tilde{z}^{r}+w^{r}(p)$ for each
$r=1,\cdots,n,$ then these coordinates satisfy conditions (i) and (ii).
\end{proof}

Using Lemma \ref{lem-coordinate}, we can compute the metric tensors and their derivatives as following:
\begin{align*}
 g_{i\bar{j}}&=\hat{g}_{i\bar{j}}+\sum_{r=1}^{m}\left(\beta_{r}||S_{r}||_{r}^{2\beta_{r}}(\log||S_{r}||_{r}^{2})_{i}\right)_{\bar{j}}\\&=
 \hat{g}_{i\bar{j}}+\sum_{r=1}^{m}\left(\beta_{r}^{2}
\frac{(||S_{r}||_{r}^{2})_{i}(||S_{r}||_{r}^{2})_{\bar{j}}}{||S_{r}||_{r}^{4-2\beta_{r}}}
-\beta_{r}||S_{r}||_{r}^{2\beta_{r}}\Theta_{r,i\bar{j}}\right),
\end{align*}
where the last equality comes from Poincare-Lelong equation and
$\Theta_{r}=-\ddb\log a_{r}$ represents the curvature form of the line bundle $[D_{r}].$ Now for the first order
derivatives, we have
\begin{align*}
g_{i\bar{j},k}&=\hat{g}_{i\bar{j},k}+\sum_{r=1}^{m}\frac{\beta_{r}^{2}(\beta_{r}-2)(||S_{r}||_{r}^{2})_{i}(||S_{r}||_{r}^{2})_{\bar{j}}(||S_{r}||_{r}^{2})_{k}}
{||S_{r}||_{r}^{2(3-\beta_{r})}}\\&+\sum_{r=1}^{m}\beta_{r}^{2}
\frac{(||S_{r}||_{r}^{2})_{ik}(||S_{r}||_{r}^{2})_{\bar{j}}+(||S_{r}||_{r}^{2})_{i}(||S_{r}||_{r}^{2})_{\bar{j}k}}
{||S_{r}||_{r}^{2(2-\beta_{r})}}\\&-\sum_{r=1}^{m}\left(\beta_{r}^{2}
\frac{(||S_{r}||_{r}^{2})_{k}}{||S_{r}||_{r}^{2(1-\beta_{r})}}\Theta_{r,i\bar{j}}
+\beta_{r}||S_{r}||_{r}^{2\beta_{r}}\Theta_{r,i\bar{j},k}\right),
\end{align*}
Note that $||S_{r}||_{r}^{2}=a_{r}|z_{r}|^{2},$ using (i) of Lemma \ref{lem-coordinate}, we can obtain the second order derivatives at $p$:
\begin{align*}
g_{i\bar{j},k\bar{l}}(p)=&\hat{g}_{i\bar{j},k\bar{l}}+\sum_{r=1}^{m}\frac{\beta_{r}^{2}(\beta_{r}-1)^{2}
}{|z_{r}|^{2(2-\beta_{r})}}\delta_{ri}\delta_{rj}\delta_{rk}\delta_{rl}\\+&\sum_{r=1}^{m}
\frac{\beta_{r}^{3}}{|z_{r}|^{2(1-\beta_{r})}}
(a_{r,i\bar{j}}\delta_{rk}\delta_{rl}+a_{r,i\bar{l}}\delta_{rk}\delta_{rj}+a_{r,k\bar{j}}\delta_{ri}\delta_{rl}+
a_{r,k\bar{l}}\delta_{ri}\delta_{rj})\\+&\sum_{r=1}^{m}\beta_{r}^{2}|z_{r}|^{2\beta_{r}}
\left((a_{r,ik\bar{l}}\delta_{rj}+a_{r,i\bar{j}k}\delta_{rl})z_{r}+(a_{r,i\bar{j}\bar{l}}\delta_{rk}+
a_{r,\bar{j}k\bar{l}}\delta_{ri})\bar{z}_{r}\right)\\+&\sum_{r=1}^{m}\beta_{r}^{2}|z_{r}|^{2\beta_{r}}
(a_{r,i\bar{j}}a_{r,k\bar{l}}+a_{r,i\bar{l}}a_{r,k\bar{j}})|z_{r}|^{2}
-\sum_{r=1}^{m}\beta_{r}\Theta_{r,i\bar{j},k\bar{l}}.
\end{align*}
Meanwhile we can simplify the expressions of metric tensors and first order derivatives at point $p$:
\begin{align*}
g_{i\bar{j}}(p)&=\hat{g}_{i\bar{j}}+\sum_{r=1}^{m}\left(\frac{\beta_{r}^{2}\delta_{ri}\delta_{rj}}
{|z_{r}|^{2(1-\beta_{r})}}+\beta_{r}^{2}|z_{r}|^{2\beta_{r}}a_{r,i\bar{j}}\right)\\
g_{i\bar{j},k}(p)&=\hat{g}_{i\bar{j},k}+\sum_{r=1}^{m}\frac{\beta_{r}^{2}(\beta_{r}-1)}
{|z_{r}|^{2(2-\beta_{r})}}\bar{z}_{r}\delta_{ri}\delta_{rj}\delta_{rk}\\&+\sum_{r=1}^{m}
\frac{\beta_{r}^{2}\bar{z}_{r}}{|z_{r}|^{2(1-\beta_{r})}}(a_{r,i\bar{j}}\delta_{rk}
+a_{r,k\bar{j}}\delta_{ri})+\sum_{r=1}^{m}\beta_{r}^{2}|z_{r}|^{2\beta_{r}}a_{r,i\bar{j}k}.
\end{align*}
For $1\leq r,s\leq m, r\neq s$ and $m+1\leq t,t'\leq n,$ we can easily have
\begin{equation}\label{eq:metric inverse}
g^{r\bar{t}}(p)=O(|z_{r}|^{2(1-\beta_{r})}),\;g^{t\bar{t'}}=O(1).
\end{equation}
We need the following lemma to give a more precise estimate for $g^{r\bar{s}}(p)(1\leq r,s\leq m):$
\begin{lemma}\label{lem-diagonal}
\begin{align}
g^{r\bar{r}}(p)=\frac{\beta_{r}^{-2}|z_{r}|^{2(1-\beta_{r})}}{1+c_{r}(p)|z_{r}|^{2(1-\beta_{r})}}
+O(|z_{r}|^{4(1-\beta_{r})}\sum_{s=1}^{m}(|z_{s}|^{2(1-\beta_{s})}+|z_{s}|^{2\beta_{s}})),\label{eq:metric inverse-diagonal}\\
g^{r\bar{s}}(p)=\beta_{r}^{-2}\beta_{s}^{-2}|z_{r}|^{2(1-\beta_{r})}|z_{s}|^{2(1-\beta_{s})}((-1)^{r+s}\hat{g}_{s\bar{r}}+o(1))\;(r\neq s),
\label{eq:metric inverse-upper}
\end{align}
where $c_{r}(p):=\beta_{r}^{-2}\frac{det(\hat{g}_{i\bar{j}})_{i,j=r,m+1,\cdots,n}}{det(\hat{g}_{i\bar{j}})_{i,j=m+1,\cdots,n}}(p)$
and $0<C_{1}<c_{r}(p)<C_{2}$ for all $p\in M$ and $r=1,\cdots,m.$
\end{lemma}
\begin{proof}
Actually we only need to prove the result for $r=1,s=2.$ For simplicity we denote $$a_{i\bar{j}}:=\hat{g}_{i\bar{j}}(p)
+\sum_{r=1}^{m}\beta_{r}^{2}||S_{r}||_{r}^{2\beta_{r}}a_{r,i\bar{j}},\;b_{r}=\frac{\beta_{r}^{2}}{|z_{r}|^{2(1-\beta_{r})}},$$
then we know that $g_{i\bar{j}}=a_{i\bar{j}}+b_{i}\delta_{ij}$ where $i\leq m.$ By determinant rule we know that
$g^{1\bar{1}}(p)=\frac{G_{1\bar{1}}}{det g}$ where $G_{1\bar{1}}$ represents the cofactor of $g_{1\bar{1}}.$ Let us
compute $G_{1\bar{1}}$ and $det g$ respectively. We can denote $A_{r_{1},r_{2},\cdots,r_{k},R}$ as the (k+n-m)-th
minor $det(a_{i_{p}\bar{i}_{q}})_{r_{1},\cdots r_{k},m+1,\cdots,n},$ where $1\leq r_{1}<\cdots<r_{k}\leq m.$ Now
make use of determinant rules, we can have such decomposition of $det\;g:$
\begin{align*}
det\;g=&
\begin{vmatrix}
b_{1}        & 0                 & \cdots  &  0           \\
a_{2\bar{1}} &a_{2\bar{2}}+b_{2} & \cdots  &  a_{2\bar{n}}\\
             &                   & \cdots  &              \\
a_{n\bar{1}} &                   & \cdots  &  a_{n\bar{n}}\\
\end{vmatrix}
+
\begin{vmatrix}
a_{1\bar{1}} & a_{1\bar{2}}      & \cdots  &  a_{1\bar{n}}\\
a_{2\bar{1}} &a_{2\bar{2}}+b_{2} & \cdots  &  a_{2\bar{n}}\\
             &                   & \cdots  &              \\
a_{n\bar{1}} &                   & \cdots  &  a_{n\bar{n}}\\
\end{vmatrix}
\\=&
\begin{vmatrix}
b_{1}        & 0                 & \cdots  &  0           \\
0            & b_{2}             & \cdots  &  0           \\
             &                   & \cdots  &              \\
a_{n\bar{1}} &                   & \cdots  &  a_{n\bar{n}}\\
\end{vmatrix}
+
\begin{vmatrix}
b_{1}        & 0                 & \cdots  &  0           \\
a_{2\bar{1}} &a_{2\bar{2}}       & \cdots  &  a_{2\bar{n}}\\
             &                   & \cdots  &              \\
a_{n\bar{1}} &                   & \cdots  &  a_{n\bar{n}}\\
\end{vmatrix}
\\+&
\begin{vmatrix}
a_{1\bar{1}} & a_{1\bar{2}}      & \cdots  &  a_{1\bar{n}}\\
0            & b_{2}             & \cdots  &  0           \\
             &                   & \cdots  &              \\
a_{n\bar{1}} &                   & \cdots  &  a_{n\bar{n}}\\
\end{vmatrix}
+
\begin{vmatrix}
a_{1\bar{1}} & a_{1\bar{2}}      & \cdots  &  a_{1\bar{n}}\\
a_{2\bar{1}} &a_{2\bar{2}}       & \cdots  &  a_{2\bar{n}}\\
             &                   & \cdots  &              \\
a_{n\bar{1}} &                   & \cdots  &  a_{n\bar{n}}\\
\end{vmatrix}
=\cdots+\cdots
\end{align*}
\begin{align*}
=&
\begin{vmatrix}
b_{1}          & 0                 & \cdots  & 0          & \cdots& 0  \\
0              & b_{2}             & \cdots  & 0          & \cdots& 0  \\
               &                   & \ddots  &            &       &    \\
0              &                   & \cdots  & b_{m}      &       & 0  \\
a_{m+1,\bar{1}}&                   & \cdots  &            &       & a_{m+1,\bar{n}}\\
               &                   & \cdots  &            &       &    \\
a_{n\bar{1}}   &                   & \cdots  &            &       &a_{n\bar{n}}\\
\end{vmatrix}
\end{align*}
\begin{align*}
+&\left(
\begin{vmatrix}
b_{1}          & 0                 & \cdots  & 0          & \cdots& 0  \\
0              & b_{2}             & \cdots  & 0          & \cdots& 0  \\
               &                   & \ddots  &            &       &    \\
0              &                   & \cdots  & b_{m-1}    &       & 0  \\
a_{m,\bar{1}}  &                   & \cdots  &            &       & a_{m,\bar{n}}\\
               &                   & \cdots  &            &       &    \\
a_{n\bar{1}}   &                   & \cdots  &            &       &a_{n\bar{n}}\\
\end{vmatrix}
+\cdots+
\begin{vmatrix}
a_{1\bar{1}}   & a_{1\bar{2}}      & \cdots  & 0          & \cdots& a_{1\bar{n}}  \\
0              & b_{2}             & \cdots  & 0          & \cdots& 0  \\
               &                   & \ddots  &            &       &    \\
0              &                   & \cdots  & b_{m}      &       & 0  \\
a_{m+1,\bar{1}}&                   & \cdots  &            &       & a_{m+1,\bar{n}}\\
               &                   & \cdots  &            &       &    \\
a_{n\bar{1}}   &                   & \cdots  &            &       &a_{n\bar{n}}\\
\end{vmatrix}
\right)
\\+&\cdots+
\begin{vmatrix}
a_{1\bar{1}} & a_{1\bar{2}}      & \cdots  &  a_{1\bar{n}}\\
             &                   & \cdots  &              \\
a_{m\bar{1}} &                   & \cdots  &  a_{m\bar{n}}\\
             &                   & \cdots  &              \\
a_{n\bar{1}} &                   & \cdots  &  a_{n\bar{n}}\\
\end{vmatrix}
\\=&b_{1}b_{2}\cdots b_{m}A_{R}+b_{1}b_{2}\cdots b_{m}\sum_{r=1}^{m}\frac{A_{r,R}}{b_{r}}+b_{1}b_{2}\cdots
b_{m}\sum_{1\leq r<s}^{m}\frac{A_{r,s,R}}{b_{r}b_{s}}+\cdots+det(a_{i\bar{j}}).
\end{align*}
Similarly, we have that
$$G_{1\bar{1}}=b_{2}b_{3}\cdots b_{m}\left(A_{R}+\sum_{r=2}^{m}\frac{A_{r,R}}{b_{r}}+\sum_{2\leq r<s}^{m}
\frac{A_{r,s,R}}{b_{r}b_{s}}\right)+\cdots+det(a_{i\bar{j}})_{2\leq i,j\leq m},$$
$$G_{1\bar{2}}=-a_{2\bar{1}}b_{3}\cdots b_{m}A_{R}+\cdots.$$
Now we can have such estimate:
\begin{align*}
 \frac{G_{1\bar{1}}}{det\;g}=b_{1}^{-1}\frac{1+\sum_{r=2}^{m}\frac{A_{r,R}}{A_{R}b_{r}}+\sum_{2\leq r<s}^{m}\frac{A_{r,s,R}}
{A_{R}b_{r}b_{s}}+\cdots}{1+\sum_{r=1}^{m}\frac{A_{r,R}}{A_{R}b_{r}}+\sum_{1\leq r<s}^{m}\frac{A_{r,s,R}}{A_{R}b_{r}b_{s}}+\cdots}
=\frac{b_{1}^{-1}}{B},
\end{align*}
where
\begin{align*}
 B&=(1+\sum_{r=1}^{m}\frac{A_{r,R}}{A_{R}b_{r}}+\sum_{1\leq r<s}^{m}\frac{A_{r,s,R}}{A_{R}b_{r}b_{s}}+\cdots)\{1-(\sum_{r=2}^{m}
\frac{A_{r,R}}{A_{R}b_{r}}+\sum_{2\leq r<s}^{m}\frac{A_{r,s,R}}{A_{R}b_{r}b_{s}}+\cdots)\\&+(\sum_{r=2}^{m}\frac{A_{r,R}}{A_{R}b_{r}}
+\sum_{2\leq r<s}^{m}\frac{A_{r,s,R}}{A_{R}b_{r}b_{s}}+\cdots)^{2}+\cdots\}\\&=1+\frac{A_{1,R}}{A_{R}b_{1}}+\frac{1}{A_{R}b_{1}}
\sum_{r=2}^{m}\frac{1}{b_{r}}(A_{1,r,R}-\frac{A_{1,R}A_{r,R}}{A_{R}})+O(\frac{1}{b_{1}}\sum_{r=2}^{m}\frac{1}{b_{r}^{2}}).
\end{align*}
Note that $$a_{i\bar{j}}:=\hat{g}_{i\bar{j}}(p)+\sum_{s=1}^{m}\beta_{s}^{2}|z_{s}|^{2\beta_{s}}a_{s,i\bar{j}}
=\hat{g}_{i\bar{j}}(p)+O(\sum_{s=1}^{m}|z_{s}|^{2\beta_{s}}),$$ \eqref{eq:metric inverse-diagonal} follows. \eqref{eq:metric inverse-upper}
follows from the computation of $G_{1\bar{2}}$ more easily.
\end{proof}

Take two unit vectors $\eta=\eta^{i}\frac{\partial}{\partial z_{i}},\nu=\nu^{i}\frac{\partial}{\partial z_{i}}\in T_{p}^{1,0}M,$
so that $g(\eta,\eta)|_{p}=g(\nu,\nu)|_{p}=1.$ Then from the expression of $g_{i\bar{j}}$ we have
\begin{equation}\label{eq:unit vector}
 \eta^{r},\nu^{r}=O(|z_{r}|^{1-\beta_{r}}),\eta^{t},\nu^{t}=O(1)\;for\;r=1,\cdots,m,\;t=m+1,\cdots,n.
\end{equation}
By the definition of bisectional curvature, we set
$$Bisec_{\dr}(\eta,\nu)=R(\eta,\bar{\eta},\nu,\bar{\nu})=R_{i\bar{j}k\bar{l}}\eta^{i}\bar{\eta^{j}}\nu^{k}\bar{\eta^{l}}=
\sum_{i,j,k,l}(\Lambda_{i\bar{j}k\bar{l}}+\Pi_{i\bar{j}k\bar{l}}),$$
with $\Lambda_{i\bar{j}k\bar{l}}:=-g_{i\bar{j},k\bar{l}}\eta^{i}\bar{\eta^{j}}\nu^{k}\bar{\eta^{l}},$ and $\Pi_{i\bar{j}k\bar{l}}
:=g^{p\bar{q}}g_{i\bar{q},k}g_{p\bar{j},\bar{l}}\eta^{i}\bar{\eta^{j}}\nu^{k}\bar{\eta^{l}}$ (no summations on $i,j,k,l$). By
\eqref{eq:metric inverse}-\eqref{eq:unit vector} we have
$|\Lambda_{i\bar{j}k\bar{l}}|\leq C$ except for the terms $\sum\limits_{r=1}^{m}\Lambda_{r\bar{r}r\bar{r}},$ hence
\begin{align}\label{eq:bisec-1}
 \sum_{i,j,k,l}\Lambda_{i\bar{j}k\bar{l}}(p)=&O(1)-\sum_{r=1}^{m}\frac{\beta_{r}^{2}(\beta_{r}-1)^{2}}{|z_{r}|^{2(2-\beta_{r})}}|\eta^{r}|^{2}|\nu^{r}|^{2}.
\end{align}

Now we can deal with the first case in Theorem \ref{thm-upper}. For this case We have the following lemma:
\begin{lemma}\label{lem-bisec-2}
 In case that either no three irreducible divisors intersect or all angles $\beta_{i}\leq\frac{1}{2},$ there exists a
 uniform constant $C>0$ such that for every $p\in M,$
\begin{equation}\label{eq:bisec2}
 \sum_{i,j,k,l}\Pi_{i\bar{j}k\bar{l}}(p)\leq C+\sum_{r=1}^{m}\frac{\beta_{r}^{2}(\beta_{r}-1)^{2}}{|z_{r}|^{2(2-\beta_{r})}}|\eta^{r}|^{2}|\nu^{r}|^{2}.
\end{equation}
\end{lemma}
\begin{proof}
By Brendle's computation in \cite{Br}, we can easily bound all the terms if $\beta_{i}\leq\frac{1}{2}$ for all i. Now we consider the general case.
As lemma A.3 in \cite{JMR}, we define a bilinear Hermitian form of two tensors $a=[a_{i\bar{q}k}],b=[b_{j\bar{p}l}]\in(\C)^{3}$ satisfying
$a_{i\bar{q}k}=a_{k\bar{q}i},b_{j\bar{p}l}=b_{l\bar{p}j}$ by setting $$\langle[a_{i\bar{q}k}],[b_{j\bar{p}l}]\rangle:=
\sum_{i,j,k,l,p,q}g^{p\bar{q}}(\eta^{i}g_{i\bar{q},k}\nu^{k})(\overline{\eta^{j}g_{j\bar{p},l}\nu^{l}}).$$
Obviously it is a nonnegative bilinear form. We denote by $||\cdot||$ the associated norm. Then $\sum\limits_{i,j,k,l}
\Pi_{i\bar{j}k\bar{l}}=||[a_{i\bar{j}k}]||^{2}.$ We write $$g_{i\bar{j},k}=A_{i\bar{j}k}+B_{i\bar{j}k}+D_{i\bar{j}k}+E_{i\bar{j}k}$$
with $$A_{i\bar{j}k}:=\hat{g}_{i\bar{j},k},\;B_{i\bar{j}k}:=\sum_{r=1}^{m}\beta_{r}^{2}|z_{r}|^{2\beta_{r}}a_{r,i\bar{j}k},$$
$$D_{i\bar{j}k}:=\sum_{r=1}^{m}\frac{\beta_{r}^{2}\bar{z}_{r}}{|z_{r}|^{2(1-\beta_{r})}}(a_{r,i\bar{j}}\delta_{rk}
+a_{r,k\bar{j}}\delta_{ri}),$$
$$E_{i\bar{j}k}:=\sum_{r=1}^{m}\frac{\beta_{r}^{2}(\beta_{r}-1)}
{|z_{r}|^{2(2-\beta_{r})}}\bar{z}_{r}\delta_{ri}\delta_{rj}\delta_{rk}.$$
Denote $A:=[A_{i\bar{q}k}]$ and similarly $B,D,E.$ Using \eqref{eq:metric inverse}, we can bound $||A+B+D||^{2}$ easily. For the crossing terms of A, B, D and E, we
have that
\begin{align*}
 2Re\langle A,E\rangle&=2Re(\sum_{i,j,k,r}g^{r\bar{j}}\hat{g}_{i\bar{j},k}\overline{E_{r\bar{r}r}})\\&\leq C\sum_{i,j,k}|
\hat{g}_{i\bar{j},k}|^{2}+\delta\sum_{r}|z_{r}|^{2(1-\beta_{r})}||E_{r\bar{r}r}||^{2}\leq C+\delta\sum_{r}|z_{r}|^{4(1-\beta_{r})}\hat{g}_{r\bar{r}}(p)
\beta_{r}^{-2}|E_{r\bar{r}r}|^{2},
\end{align*}
 where $\delta$ can be chosen small enough. By the same argument, we also have that
$$2Re\langle B+D,E\rangle\leq C+\delta\sum_{r}|z_{r}|^{4(1-\beta_{r})}\hat{g}_{r\bar{r}}(p)\beta_{r}^{-2}|E_{r\bar{r}r}|^{2}.$$
Now let us consider $||E||^{2}:$
In case that at most two divisors intersect transversely, we can take $m=2.$ As there exists a uniform constant $c_{0}<1$ such that $|\hat{g}_{1\bar{2}}|^{2}(p)\leq c_{0}^{2}\hat{g}_{1\bar{1}}(p)\hat{g}_{2\bar{2}}(p),$
By \eqref{eq:metric inverse-diagonal} and \eqref{eq:metric inverse-upper}, we have that
\begin{align*}
 ||E||^{2}(p)&=\sum_{r=1}^{2}g^{r\bar{r}}|E_{r\bar{r}r}|^{2}+\sum_{r\neq s}^{2}g^{r\bar{s}}E_{r\bar{r}r}\overline{E_{s\bar{s}s}}\\
&\leq\sum_{r=1}^{2}\frac{\beta_{r}^{-2}|z_{r}|^{2(1-\beta_{r})}}{1+c_{r}(p)|z_{r}|^{2(1-\beta_{r})}}|E_{r\bar{r}r}|^{2}+\sum_{r=1}^{2}c_{0}\hat{g}_{r\bar{r}}(p)
\beta_{r}^{-4}|z_{r}|^{4(1-\beta_{r})}|E_{r\bar{r}r}|^{2}\\&=\sum_{r=1}^{2}\beta_{r}^{-2}|z_{r}|^{2(1-\beta_{r})}\left(\frac{1+c_{0}\hat{g}_{r\bar{r}}(p)
\beta_{r}^{-2}|z_{r}|^{2(1-\beta_{r})}}{1+c_{r}(p)|z_{r}|^{2(1-\beta_{r})}}+o(1)\right)|E_{r\bar{r}r}|^{2}.
\end{align*}
Add these estimates together when $m=2$ or all cone angles $\beta_{r}\in (0,\frac{1}{2}),$ we obtain that
$$\sum_{i,j,k,l}\Pi_{i\bar{j}k\bar{l}}(p)\leq C+\sum_{r=1}^{m}\beta_{r}^{-2}|z_{r}|^{2(1-\beta_{r})}\left(\frac{1+(c_{0}+2\delta)\hat{g}_{r\bar{r}}(p)
\beta_{r}^{-2}|z_{r}|^{2(1-\beta_{r})}}{1+c_{r}(p)|z_{r}|^{2(1-\beta_{r})}}+o(1)\right)|E_{r\bar{r}r}|^{2}$$
As in our coordinate system $c_{r}(p)=\beta_{r}^{-2}\frac{A_{r,R}}{A_{R}}=\beta_{r}^{-2}\hat{g}_{r\bar{r}}(p),$ by choosing $\delta>0$ such that $2\delta+c_{0}<1,$ we obtain the lemma.
\end{proof}

 For triple or higher multiple singularities, i.e.$m>2,$ generally we do not have that the inequality
\begin{align*}
\begin{bmatrix}
  0                   & -\hat{g}_{1\bar{2}} & \cdots  &  (-1)^{m+1}\hat{g}_{1\bar{m}}\\
  -\hat{g}_{2\bar{1}}  &  0                 & \cdots  &  (-1)^{m+2}\hat{g}_{2\bar{m}}\\
                      &                    & \ddots  &                    \\
  (-1)^{m+1}\hat{g}_{m\bar{1}}  &                    & \cdots  &  0                 \\
\end{bmatrix}
<
\begin{bmatrix}
\hat{g}_{m\bar{m}} &  0                  & \cdots  &  0\\
0                  &\hat{g}_{m\bar{m}}   & \cdots  &  0\\
                   &                     & \ddots  &   \\
0                  &  0                  & \cdots  &  \hat{g}_{m\bar{m}}\\
\end{bmatrix}
\end{align*}
That is why we cannot control the crossing terms in $||E||^{2}$ by its diagonal terms so well as $m=2.$ Fortunately this observation implies that we can increase
the diagonal terms of $\hat{g}$ such that this matrix inequality holds for the modified metric. In applications to geometric problems we only change the metric in the same cohomology class.
One natural idea is to consider the new background metric with the following form
\begin{equation}\label{eq:new bkgd}
\dr'_{0}=\dr_{0}+\ddb\sum\limits_{i=1}^{m}\varphi_{r}(||S_{r}||_{r}^{2}),
\end{equation}
where $\varphi_{r}(||S_{r}||_{r}^{2})$ behaves as $\lambda_{r}||S_{r}||_{r}^{2}$ for some suitable $\lambda_{r}>0$
near $D_{r}=(S_{r}=0)$ and tends to 0 far away such that it preserves the positivity of the new background metric. This can be done by suitable cutoff argument. To see how much this modification changes the bisectional curvature we first compute the corresponding derivatives of the new background metric tensors (we denote $\hat{g}'$ as the metric tensors of $\dr'_{0}$) as in the previous section:
\begin{align*}
\hat{g}_{i\bar{j}}'=&\hat{g}_{i\bar{j}}+\sum_{r=1}^{m}\varphi'_{r}(||S_{r}||_{r}^{2})(a_{r}\delta_{ir}\delta_{jr}+a_{r,i}\delta_{jr}z_{r}
+a_{r,\bar{j}}\delta_{ir}\bar{z}_{r}+a_{r,i\bar{j}}|z_{r}|^{2})\\+&\sum_{r=1}^{m}\varphi''_{r}(||S_{r}||_{r}^{2})(a_{r}\delta_{ir}\bar{z}_{r}+
a_{r,i}|z_{r}|^{2})(a_{r}\delta_{jr}z_{r}+a_{r,\bar{j}}|z_{r}|^{2}).
\end{align*}
At the chosen point $p,$ by our assumption that $\varphi_{r}(t)=\lambda_{r}t$ when $t$ is small and the chosen coordinate system,  we get that $$\hat{g}_{i\bar{j}}'(p)=\hat{g}_{i\bar{j}}(p)+\sum_{r=1}^{m}\lambda_{r}(\delta_{ir}\delta_{jr}+O(|z_{r}|)),$$ Meanwhile by the order of the error terms we can easily construct cut-off functions so that after modification the new metric is uniformly equivalent to the original metric.
And similarly, we can also get that
$$\hat{g}_{i\bar{j},k}'(p)=\hat{g}_{i\bar{j},k}(p)+O(1),\;\hat{g}_{i\bar{j},k\bar{l}}'(p)=\hat{g}_{i\bar{j},k\bar{l}}(p)+O(1).$$
Using these estimates in the computation of the previous section, we find that almost any estimates do not change except for the estimate of
$||E||^{2}$, due to the change of the background metric tensors. Let us rewrite this formula for the modified metric:
\begin{align*}
||E||^{2}(p)&=\sum_{r=1}^{m}g'^{r\bar{r}}|E_{r\bar{r}r}|^{2}+\sum_{r\neq s}^{m}g'^{r\bar{s}}E_{r\bar{r}r}\overline{E_{s\bar{s}s}}\\
&\leq\sum_{r=1}^{m}\frac{\beta_{r}^{-2}|z_{r}|^{2(1-\beta_{r})}}{1+c'_{r}(p)|z_{r}|^{2(1-\beta_{r})}}|E_{r\bar{r}r}|^{2}
\\&+
\beta_{r}^{-2}\beta_{s}^{-2}|z_{r}|^{2(1-\beta_{r})}|z_{s}|^{2(1-\beta_{s})}(-1)^{r+s}(\hat{g}_{s\bar{r}}'+o(1))E_{r\bar{r}r}\overline{E_{s\bar{s}s}}.
\end{align*}
As in the lemma \ref{lem-bisec-2}, we need to control the second term by the diagonal terms. Now we could choose $\lambda_{r}$ large enough,
$r=1,\cdots,m$ so that for some constant $c'_{0}<1$ the following inequality holds for all points lying in some tubular neighborhood of the
simple normal crossing divisor $D:$
\begin{align*}
\begin{bmatrix}
  0                   & -\hat{g}_{1\bar{2}}' & \cdots  &  (-1)^{m+1}\hat{g}_{1\bar{m}}'\\
  -\hat{g}_{2\bar{1}}'  &  0                 & \cdots  &  (-1)^{m+2}\hat{g}_{2\bar{m}}'\\
                      &                    & \ddots  &                    \\
  (-1)^{m+1}\hat{g}_{m\bar{1}}'  &                    & \cdots  &  0                 \\
\end{bmatrix}
<c_{0}'
\begin{bmatrix}
\hat{g}_{1\bar{1}}' &  0                  & \cdots  &  0\\
0                  &\hat{g}_{2\bar{2}}'   & \cdots  &  0\\
                   &                     & \ddots  &   \\
0                  &  0                  & \cdots  &  \hat{g}_{m\bar{m}}'\\
\end{bmatrix}
\end{align*}
Then the corresponding estimate will follow as above and the following proposition holds:

\begin{proposition}\label{prop-bisec-3}
In general situations (triple or higher multiple singularities with arbitrary cone angles), there exist a smooth function
$\varphi_{0}=\sum\limits_{i=1}^{m}\varphi_{r}(||S_{r}||_{r}^{2})$ where $\varphi_{r}(t)=\lambda_{r}t$ for some enough large constants
$\lambda_{r},r=1,\cdots,m$ near 0 and vanish when t is larger, such that there exists a
 uniform constant $C>0$ such that for every $p\in M,$ and new metric $\dr'_{\epsilon}=\dr_{\epsilon}+\ddb\varphi,$
\begin{equation}\label{eq:bisec3}
 \sum_{i,j,k,l}\Pi_{i\bar{j}k\bar{l}}(p)\leq C+\sum_{r=1}^{m}\frac{\beta_{r}^{2}(\beta_{r}-1)^{2}}{|z_{r}|^{2(2-\beta_{r})}}|\eta^{r}|^{2}|\nu^{r}|^{2}.
\end{equation}
\end{proposition}
Combine Lemma \ref{lem-bisec-2} and this proposition, the proof of Theorem \ref{thm-upper} is finished.\\

As a possible application of this curvature estimate, we remark that if the linear theory in \cite{Do} \cite{JMR} could be established for
simple normal crossing divisor, this curvature estimate could enable us to establish A priori Laplacian estimate for the solutions along the continuity path.
For simplicity we only establish the A priori estimate for the solution to the conical \ka-Einstein equation \eqref{eq:cke-cma}. We denote
$$\dr_{bg}:=\dr_{0}+\epsilon\sum_{r=1}^{m}\ddb||S_{r}||^{2\beta_{r}}+\ddb\varphi_{0}$$ as the background conic metric whose bisectional curvature has an upper bound $\Lambda$ on $M\setminus D,$ where $\epsilon$ is small and $\varphi_{0}\in PSH(M,\dr_{0})\cap C^{\infty}(M).$ Actually it is equivalent to the standard conic metric $\dr_{cone}$ defined in \eqref{eq:def-conic}. For any $\delta>0,$ put
$$Q:=\log(\prod_{r=1}^{m}||S_{r}||^{2\delta}tr_{\dr_{\varphi}}\dr_{bg})-
A(\varphi-\varphi_{0}-\epsilon\sum_{r=1}^{m}||S_{r}||^{2\beta_{r}}).$$  On $M\setminus D,$ as $\dr_{\varphi}$ has positive Ricci curvature and the bisectional curvature of $\dr_{bg}$ is less than $\Lambda,$ using Chern-Lu's Inequality we immediately have that
$$\Delta Q\geq (A-\Lambda)tr_{\dr_{\varphi}}\dr_{bg}+\delta\sum_{r=1}^{m} tr_{\dr_{\varphi}}R(||\cdot||_{r})-An.$$
Take $A=\Lambda+1,$ as an assumption for the A priori bound, we know that $tr_{\dr_{\varphi}}\dr_{bg}$ is bounded and moreover satisfies H\"older continuity near the divisor,  we know that the maximal of $Q$ on $M$ is attained at $p\in M\setminus D,$ where $$tr_{\dr_{\varphi}}\dr_{bg}(p)\leq-\delta\sum_{r=1}^{m} tr_{\dr_{\varphi}}R(||\cdot||_{r})+(\Lambda+1)n\leq(\Lambda+1)n,$$ by maximal principle. As $\varphi,\varphi_{0}$ are uniformly bounded on $M$ we conclude that it holds that $tr_{\dr_{\varphi}}\dr_{bg}\leq\frac{C_{0}}{\prod_{r=1}^{m}||S_{r}||^{2\delta}}.$ Let $\delta$ tend to 0 we prove that $$\dr_{\varphi}\geq C_{1}\dr_{bg}.$$ By the conical Monge-Ampere equation \eqref{eq:cke-cma} it is easy to
see that $\dr_{\varphi}^{n}$ is equivalent to $\dr_{bg}^{n}$ so $$\dr_{\varphi}\leq C_{2}\dr_{bg}.$$ Thus this estimate could generalize the original Laplacian estimate in \cite{JMR} to simple normal crossing divisor case.

\section{Further discussions}

By the proof of Theorem \ref{mainthm} we notice that the existence of \ka-Einstein metric guarantees the properness of the smooth Ding functional or Mabuchi functional by \cite{Ti97}, which is used to derive the properness of twisted functionals. In this sense, under the assumptions of divisors, the existence of one conical \ka-Einstein metric implies the existence of conical \ka-Einstein metrics with smaller cone angles and the constant $\mu:$

\begin{theorem}\label{thm-modify}
Given a Fano manifold $(M,\dr_{0})$ which satisfies all the assumptions in Theorem\ref{mainthm} except for replacing the existence of
\ka-Einstein metric by the existence of one conical \ka-Einstein metric $\dr_{\mu_{0}}\in c_{1}(M)$ such that
$$Ric(\dr_{\mu_{0}})=\mu_{0}\dr_{\mu_{0}}+\sum_{r=1}^{m}2\pi(1-\beta_{r,0})[D_{r}]$$
then for $\mu\in(0,\mu_{0}),$ corresponding conical \ka-Einstein metrics still exist.
\end{theorem}
Note that the properness of the twisted Ding functional corresponding to $\mu_{0}$ could be derived from a
modification of \cite{TZ} and we want to thank Professor X. H. Zhu for pointing out this to us.

Another interesting problem is the requirement of the coefficients $c_{r},\lambda_{r}.$ In one smooth divisor case, by \cite{LS} \cite{SW} such similar requirement is necessary due to the obstruction of log-K stability and
the example of $\mathbb{P}^{2}$ with a quadratic curve shows this point. But in the case of simple normal crossing divisor with $m\geq 2,$ such obstruction will be much more subtle, e.g, toric examples constructed in \cite{SW}. We hope to investigate and analyze more general cases in the future.

\end{document}